\documentclass[a4paper,12pt]{amsart}
\usepackage{amsmath,amsthm,amssymb}
\usepackage{enumerate}
\newcommand\mQ{\mathcal Q}  \newcommand\mS{\mathcal S} 
 \newcommand\M{\mathcal M}
 \newcommand\ep{\epsilon} \newcommand\fC{\mathfrak C}
\newcommand\mC{\mathcal C}  \newcommand\mO{\mathcal O}
 \newcommand\mG{\mathcal G}
\newcommand\mH{\mathcal H} \newcommand\mV{\mathcal V} 
 \newcommand\mL{\mathcal L}
\newcommand\bP{\mathbb P} \newcommand\bG{\mathbb G}
\newcommand\bF{\mathbb F} \newcommand\bZ{\mathbb Z}
\newcommand{\bbsm}{\left (\begin{smallmatrix}}      \newcommand{\besm}{\end{smallmatrix}\right )}
\newcommand{\bsm}{\left[ \begin{smallmatrix}}      \newcommand{\esm}{\end{smallmatrix}\right]}

\newcommand{\bbm}{\left [\begin{matrix}}      \newcommand{\bem}{\end{matrix}\right ]}
\newcommand{\beq}{\begin{equation}}      \newcommand{\eeq}{\end{equation}}
\newcommand{\beqn}{\begin{eqnarray}}      \newcommand{\eeqn}{\end{eqnarray}}
\newcommand{\beqs}{\begin{eqnarray*}}      \newcommand{\eeqs}{\end{eqnarray*}}
\newcommand{\bep}{\begin{proof}}      \newcommand{\eep}{\end{proof}}

\newcommand{\stirling}[2]{\genfrac{[}{]}{0pt}{1}{#1}{#2}}

\newcommand{\Stirling}[2]{\genfrac{[}{]}{0pt}{0}{#1}{#2}}

\newtheorem*{problem}{Problem}
\newtheorem{theorem}{Theorem}[section]
\newtheorem{lemma}[theorem]{Lemma}

\newtheorem{definition}[theorem]{Definition}

\begin{document} 
\title{Higher weight spectra of ternary codes associated to the quadratic Veronese $3$-fold}
\author{Krishna Kaipa}
\author{Puspendu Pradhan }
\address{Department of Mathematics, Indian Institute of Science Education and Research, Pune, Maharashtra, 411008 India.}
\email{kaipa@iiserpune.ac.in, puspendupradhan1@gmail.com}



\begin{abstract} 
The problem studied in this work is to determine the higher weight spectra of the Projective Reed-Muller codes associated to the Veronese $3$-fold $\mV$ in $PG(9,q)$, which is the image of the quadratic Veronese embedding of $PG(3,q)$ in $PG(9,q)$. We reduce the problem to the following combinatorial problem in finite geometry:   For each subset $S$ of  $\mathcal V$, determine the dimension of the linear subspace of $PG(9,q)$ generated  by $S$.
We develop a systematic method to solve the latter problem. We implement the method  for $q=3$, and use it to obtain the higher weight spectra of the associated code. The case of a general finite field $\bF_q$ will be treated in a future work. 
\end{abstract}

\maketitle

\section{Introduction} \label{intro}
Let $PG(n,q)$ denote the $n$-dimensional projective space defined over a finite field $\bF_q$.  Let $K =(n^2+3n+2)/2$. The  quadratic Veronese $n$-fold $\mV_n$   in $PG(K-1,q)$ is the image of the 
quadratic embedding \[ \imath : PG(n,q) \hookrightarrow PG(K-1,q), \]
 defined with respect to coordinates $(x_0, \dots, x_n)$ on $PG(n,q)$ and $(z_0, \dots, z_{K-1})$ on $PG(K-1,q)$ by  
\[\imath(x_0,\dots,x_n)=(x_0^2,x_0x_1\dots, x_{n-1}x_n,x_n^2),\]
for some fixed ordering of the $K$ monomials of degree $2$ in the variables $x_0, \dots,x_n$.
Let   $P_1, \dots, P_N$  where $N = 1+q+\dots+q^n$, be a fixed ordering of the points of $PG(n,q)$ and let 
$\bG_n$ denote the matrix of size $K \times N$ defined by 
 \[ \bG_n=\bbm v_1&\dots&v_N \bem, \]
where $v_j\in \bF_q^{K}$ represents the point $\imath(P_j)$ of $PG(K-1,q)$. The linear code  of length $N$ and dimension $K$ generated by the matrix $\bG_n$ is the quadratic Veronese code $C_n$, or the second order Projective Reed Muller code of length $N$ (\cite[\S 5]{LG}).  A  codeword  of the code $C_n$  is  a  quadratic form $f$  on $PG(n,q)$  and   the weight  of a codeword is  the number of points of $PG(n,q)$ not lying on the corresponding quadric hypersurface defined by the vanishing of $f$. An $r$-dimensional subcode of $C_n$ is a linear subspace of $C_n$ generated by some $r$ linearly independent quadratic forms $f_1, \dots, f_r$ on $PG(n,q)$. The weight of this subcode is the number of points of $PG(n,q)$ which do not lie on the common zero set of $f_1, \dots, f_r$. Equivalently, $r$-dimensional subcodes of $C_n$ correspond to codimension $r$ linear subspaces $\mL$ of $PG(K-1,q)$, with the weight of the subcode being the number of points of $\mV_n$ not lying on $\mL$.  The $r$-th generalized weight enumerator (in short GWE) of $C_n$ is the polynomial
\[ W^{(r)}(Z) = \sum_{\mL \subset PG(K-1,q) \text{ of codim  $r$}}  Z^{\# \mV_n\setminus \mL}=\sum_i A_i^{(r)} Z^i,\]
where $A_i^{(r)}$ is the number of codimension $r$ linear subspaces of $PG(K-1,q)$ which intersect $\mV_n$ in 
$N-i$ points. \\

We also explain the  algebraic-geometric significance of the polynomial $Z^NW^{(r)}(\frac{1}{Z})$. It gives the answer to the following geometric problem:
\begin{problem}
Given  natural numbers $n$ and $r$, 
let $\mS^{(r)}_n$ denote the collection of 
varieties in $\bP^n$ defined by $r$ linearly independent quadratic forms $f_1(X_0, \dots, X_n), \dots,  f_r(X_0, \dots, X_n)$  over $\bF_q$.  For $V \in \mS^{(r)}_n$, let $\#V(\bF_q)$ denote the number of $\bF_q$-rational points of $V$. Determine the generating function 
\[ \sum_{V \in \mS^{(r)}_n} Z^{\#V(\bF_q)} = \sum_{i \geq 0} c_i^{(r)}(n) Z^i,\]
where $c_i^{(r)}(n)$ is  the number of varieties in $\mS^{(r)}_n$ having $i$ rational points over $\bF_q$.
\end{problem}
The numbers $c_i^{(r)}(n)$ are just the numbers $A^{(r)}_{N-i}(C_n)$. \\

The first GWE of $C_n$ is easily calculated (see Theorem \ref{WEn}).
However, for $r>1$, the $r$-th GWEs of $C_n$ are hard to determine. In the work \cite{JV2}, Johnsen and Verdure determined the $r$-GWE's of the code $C_2$ for all $r$ and all finite fields $\bF_q$. In the work \cite{JV3}, they further obtained the $r$-th GWEs of the code $C_3$ for all $r$ over the binary field $\bF_2$. In this work, we  develop a method to determine all the $r$-GWE's of the code $C_3$ over a general finite field $\bF_q$. We now explain the gist of our method: Given a linear code $C$ of dimension $K$ with a generator matrix $M$, let $B_{j,i}$ denote the number of $K \times j$ submatrices of $M$  that have rank $i$. As shown by Jurrius-Pellikaan  \cite{JP}, the knowledge of these quantities $B_{j,i}$ for all $(j,i)$ suffices to determine the $r$-GWE's of the code $C$ for all $r$.
For the code $C_3$, the quantities $B_{j,i}$ are the number of configurations of $j$ points in $PG(3,q)$ whose images in $\mV_3$ span a $(i-1)$-dimensional linear subspace in $PG(9,q)$. We treat this problem from a finite geometry perspective, and develop a systematic method to calculate the quantities $B_{j,i}$ for all $(j,i)$, and hence  obtain the  $r$-th GWE of $C_3$ for all $r$ and for all fields $ \bF_q$.  
 We execute the method for $q=3$ in the present work, and the case of general $q$ will be treated in a future work. 
 
 \begin{theorem} \label{main_theorem} For $1 \leq r \leq 10$, let $W^{r}(Z)$ denote the $r$-th GWE of the code $C_3$ over the ternary field $\bF_3$. We have:
\begin{align*}
W^{(1)}(Z)=&2 Z^{18}(195 Z^{18} + 4212 Z^{12} + 4700 Z^9 + 5265 Z^6 + 390)\\
\begin{split}
     W^{(2)}(Z)  ={}& 13 Z^{24}(67797 Z^{16} + 243000 Z^{15} + 597780 Z^{14}
 + 933120 Z^{13}+ 1151610 Z^{12}
+ 991440 Z^{11}\\
 &+ 816480 Z^{10} + 428400 Z^9 + 182250 Z^8 + 116640 Z^7 + 43200 Z^6 + 14580 Z^4 + 760 Z^3 + 360).
\end{split}\\
\begin{split}
    W^{(3)}(Z)=
     {}&  10Z^{26}   (408114369 Z^{14} + 615347208 Z^{13} + 471425994 Z^{12} + 230935536 Z^{11} + 79511991 Z^{10}\\
         &  + 20839572 Z^9 + 5231304 Z^8 + 968760 Z^7 + 237978 Z^6 + 50544 Z^5 + 9360 Z^4 + 52 Z + 108).
\end{split}\\
 \begin{split}
    W^{(4)}(Z)={}& Z^{27}( 304888261329 Z^{13} + 152054622720 Z^{12} + 37402159860 Z^{11} + 5713746480 Z^{10} \\
&+ 641358900 Z^9 + 69893928 Z^8 + 7076160 Z^7 + 627120 Z^6 + 89505 Z^5 + 40). \end{split} \\
\begin{split}
W^{(5)}(Z)={}& 26  Z^{33}(
49159423665 Z^7 + 8099569140 Z^6 + 649260360 Z^5 + 31541400 Z^4 \\
&+ 1366565 Z^3 + 73548 Z^2 + 2430 Z + 120).  \end{split}\\
W^{(6)}(Z)=& 2 Z^{35}(237163797936 Z^5 + 12886036860 Z^4 + 334019790 Z^3 + 4956120 Z^2 + 104975 Z + 2340). \\
W^{(7)}(Z)= &10 Z^{36}(1800649161 Z^4 + 31769712 Z^3 + 252954 Z^2 + 936 Z + 13). \\
W^{(8)}(Z)= &13 Z^{38} (5557197 Z^2 + 30160 Z + 60).\\
W^{(9)}(Z)=&4 Z^{39}
(7371 Z + 10).\\
 W^{(10)}(Z) =& Z^{40}. 
\end{align*}
In particular,  the generalized Hamming weights of the code $C_3$ over $\bF_3$ are given by 
\[ d_1=18, \, d_2=24,\,d_3=26, \, d_4=27,\, d_5=33, \, d_6=35, \, d_7=36,\, d_8=38,\, d_9=39,\, d_{10}=40.\]
\end{theorem}

The rest of this paper is organized as follows. 
In \S \ref{OWE}, we record the ordinary weight enumerator of the code $C_n$. In \S \ref{Approach}, we explain how the quantities $B_{j,i}$ determine the $r$-GWE's of a code $C$, and we develop the geometric method to determine $B_{j,i}$'s for the code $C_3$. In \S \ref{r_leq_5}, we determine $B_{j,1}, B_{j,2}, B_{j,3}, B_{j,4}$ and $B_{j,5}$ for all $j$ and for a general finite field $\bF_q$.
In \S \ref{r6}, \S \ref{r7}, \S \ref{r8} and  \S \ref{r9}, we determine the quantities $B_{j,6}, B_{j,7}, B_{j,8}$ and $B_{j,9}$ respectively for all $j$ and $q=3$.
Finally, in \S \ref{BTAT} we prove Theorem \ref{main_theorem}. 

\section{Ordinary weight enumerator of the codes $C_n$} \label{OWE}
 We will write down the first generalized weight enumerator  $W^{(1)}_{C_n}(Z)$ of $C_n$. The ordinary weight enumerator of $C_n$ is just $1+(q-1) W^{(1)}_{C_n}(Z)$.
We need some notation. Let $GL_n(q)$ denote  the general linear group of invertible $n \times n$ matrices over $\bF_q$, 
The corresponding projective general linear group will be denoted $PGL_n(q)$. Let  
\beq \label{eq:gbin}
\Stirling{r}{j}_q
=\frac{\prod_{i=0}^{j-1} (q^r-q^i)}{|GL_j(q)|} = \prod_{i=0}^{j-1} \frac{(q^r-1)(q^{r-1}-1)\dots(q^{r-j+1}-1)}{(q^j-1)(q^{j-1}-1)\dots(q-1)},\eeq
denote the Gaussian binomial coefficient, which is the number of $j$-dimensional subspaces in $\bF_q^r$. Let $f_0(X,Y)$ denote a fixed quadratic form in the variables $X, Y$ which is irreducible over $\bF_q$.
We recall from \cite[Chapter 1]{GGG} that the non-degenerate quadratic forms on  $PG(2t,q)$ form a single  orbit with representative $X_0X_1 +X_2X_3+ \dots + X_{2t-2} X_{2t-1} +X_{2t}^2$.  The corresponding quadrics in $PG(2t,q)$ are known as parabolic quadrics, and  have $(q^{2t}-1)/(q-1)$ points. The size of this orbit is $q^{t^2+t} \prod_{i=1}^{t}(q^{2i+1}-1)$. The non-degenerate quadratic forms on  $PG(2t-1,q)$ form  two orbits  represented by 
$X_0X_1 +X_2X_3+ \dots +X_{2t-2} X_{2t-1}$ and $X_0X_1 +X_2X_3+ \dots X_{2t-4} X_{2t-3} +f(X_{2t-2},X_{2t-1})$. 
These orbits have sizes $\tfrac{1}{2} q^{t^2}(q^t+1) \prod_{i=1}^{t-1}(q^{2i+1}-1)$ and  $ \tfrac{1}{2} q^{t^2}(q^t-1) \prod_{i=1}^{t-1}(q^{2i+1}-1)$ respectively. The corresponding quadrics in $PG(2t-1,q)$ are known as hyperbolic and elliptic quadrics respectively, and they have  $(q^{t-1} + 1)(q^t - 1)/(q - 1)$ and $(q^{t-1} - 1)(q^t + 1)/(q - 1)$ points respectively.
Given  $m <n$ and a quadratic form $f(X_0, \dots, X_m)$ such that the corresponding quadric in $PG(m,q)$ has $\nu$ points, then the quadric defined by $f$ in $PG(n,q)$ has 
$1+ q + \dots+ q^{n-m-1} + \nu q^{n-m} $ points. Thus, the weight of  $f$ viewed as a codeword of $C_n$ is 
$q^{n-m}(1+q+ \dots+q^m- \nu)$. Since there are $\stirling{n+1}{m+1}_q$ ways to pick a $m$-dimensional linear subspace of $PG(n,q)$ we conclude that the values that weights of codewords of $C_n$ can take are as listed below:
\begin{enumerate}
\item $q^n + q^{n-t}$, where $1 \leq t \leq \lfloor (n+1)/2 \rfloor$. The number of codewords with this weight are
\[ A_{q^n + q^{n-t}}^{(1)} = \tfrac{1}{2} \stirling{n+1}{2t}_q \, q^{t^2}(q^t-1) \prod_{i=1}^{t-1}(q^{2i+1}-1).\]
\item   $q^n$,  and the number of codewords with this weight are 
\[A_{q^n }^{(1)} = \sum_{t=0}^{\lfloor n/2 \rfloor} \stirling{n+1}{2t+1}_q \,q^{t^2+t} \prod_{i=1}^{t}(q^{2i+1}-1).\]
\item $q^n - q^{n-t}$,  where $1 \leq t \leq \lfloor (n+1)/2 \rfloor$. The number of codewords with this weight are
\[ A_{q^n - q^{n-t}}^{(1)} = \tfrac{1}{2} \stirling{n+1}{2t}_q \,q^{t^2}(q^t+1) \prod_{i=1}^{t-1}(q^{2i+1}-1).\]
\end{enumerate} 
We summarize this:
\begin{theorem} \label{WEn}  The ordinary weight enumerator of the code $C_n$ is given by $1+ (q-1)W^{(1)}(C_n)$ where:
    \beq \label{eq:Cn} W^{(1)}_{C_n}(Z) = A_{q^n}^{(1)} \, Z^{q^n} + \sum_{t=1}^{\lfloor (n+1)/2 \rfloor} \left( A_{q^n - q^{n-t}}^{(1)}\,  Z^{q^n - q^{n-t}} + A_{q^n + q^{n-t}}^{(1)} \, Z^{q^n + q^{n-t}} \right),\eeq
    and the quantities $A^{(1)}_j$ are as above.
\end{theorem}

In particular, for $n=3$ we have $6$ orbits of quadratic forms on $PG(3,q)$, which we record in Table \ref{tab:table1}, including their sizes, and weights when regarded as codewords of the code $C_n$.
\begin{table}
    \centering
    \begin{tabular}{|c|c|c|c|} \hline
     \textbf{Orbit} & \textbf{Representative} &  \textbf{Size}  & \textbf{weight} \\ [0.5ex] \hline 
     &&&\\ 
     $\mO_1$ &$X_0X_1$ & $\tbinom{q^3+q^2+q+1}{2}$  & $q^3-q^2$ \\
&&&\\
$\mO_2$ &$X_0X_3-X_1X_2$ & $ \tfrac{1}{2} (q^6+q^4)(q^3-1)$& 
 $q^3-q$ \\ &&&\\
$\mO_3$&$X_0^2$ & $(1+q)(1+q^2)$ & $q^3$ \\  &&&\\
$\mO_4$ &$X_1^2-X_0X_2$ & $(1+q)(1+q^2)(q^{5}-q^2)$ & $q^3$ \\ &&&\\
$\mO_5$& $X_2X_3-f_0(X_0,X_1)$ & $\tfrac{1}{2} (q^6-q^4)(q^{3}-1)$ & $ q^3+q$\\ &&&\\
$\mO_6$&$f_0(X_0,X_1)$ & $\tfrac{1}{2} (q^3-1)(q^3+q)$ & $q^3+q^2$  \\ &&&\\
\hline
    \end{tabular}
    \caption{List of $G$-orbits of quadratic forms on $PG(3,q)$}
    \label{tab:table1}
\end{table}
Specializing \eqref{eq:Cn} to $n=3$,  we get:
 \begin{multline}
      W^{(1)}_{C_3}(Z) = 
  \tfrac{1}{2} (q^3+q^2+q+1)(q^3+q^2+q)\,  Z^{q^3 - q^{2}} + 
  \tfrac{1}{2} (q^6+q^4)(q^3-1)\,  Z^{q^3 - q}\\
  + (q^5-q^2+1)(q^3+q^2+q+1)  \, Z^{q^3}\\  + \tfrac{1}{2} (q^6-q^4)(q^3-1)\, Z^{q^3 + q} + \tfrac{1}{2} (q^3-1)(q^3+q)\, Z^{q^3 + q^{2}} 
 \end{multline}  
  Further specializing this to $q=3$, we get
 \[2 Z^{18}(195 Z^{18} + 4212 Z^{12} + 4700 Z^9 + 5265 Z^6 + 390),\]
 which matches with Theorem \ref{main_theorem}.

\section{Approach to the Problem}\label{Approach}
Let $C$ be a linear $[N,K,d]_q$ code over $\bF_q$ of length $N$, dimension $K$ and minimum distance $d$ with generator matrix $M_{K \times N}$. For $J \subset \{1, \dots,N\}$, let $r_J$ be the rank of the submatrix of $M_{K \times N}$ on the columns indexed by $J$. 
The higher weight spectrum of $C$ can be obtained just by knowing $r_J$ for each $J \subset \{1, \dots,N\}$.  To see this, consider for each $1 \leq j \leq N$, the polynomials $B_j(T) \in \bZ[T]$ defined by 

 \begin{align} \label{eq:Bj}
B_j(T)= \sum_J  (T^{K-r_J} -1),
\end{align}
where the sum runs over all subsets $J \subset \{1, \dots,N\}$ of size $j$. We define $B_0(T)=(T^K-1)$. The \emph{extended weight enumerator} polynomial (in short EWE) is the polynomial $W_C(Z;T) \in \bZ[Z,T]$ defined by 
 \begin{align} \label{eq:EWE}
 \nonumber W(Z;T)=& 1  + \sum_{j=0}^N    B_j(T) (1-Z)^j Z^{N-j} \\ 
 = &1  + (T^K-1) Z^N + \sum_{j=1}^N    B_j(T) (1-Z)^j Z^{N-j}. 
  \end{align}
Let $A_j(T) \in \bZ[T]$ be polynomials defined by
\begin{align} \label{eq:GWEP}
 W(Z;T)= 1 + \sum_{ j=1}^N  A_j(T)  Z^j.
\end{align}

The polynomial $W(Z;q^m)$ is the weight enumerator polynomial of the extended code $C \otimes \bF_{q^m}$ (see \cite[Proposition 5.23]{JP}). Thus,  $A_j(q^m)$ is the number of weight $j$ codewords of $C \otimes \bF_{q^m}$.
The $r$-th generalized weight enumerator polynomial of $C$ can be obtained from extended weight enumerator polynomial:

\begin{theorem}\cite[Theorem 5.7]{JP} \label{CO}
Let $C$ be a linear code over $\bF_q$ and $0\leq r\leq K$. Then the $r$-th generalized weight enumerator polynomial  of $C$ can be obtained from the  extended weight enumerator polynomial of $C$ by: 
\[W^{(r)}(Z)= \frac{1}{|GL_r(q)|}\sum_{j=0}^{r}{r \brack j}_q(-1)^{r-j}q^{\tbinom{r-j}{2}}W(Z;q^j).\]
\end{theorem}

  From \eqref{eq:Bj}, \eqref{eq:EWE} and Theorem \ref{CO}, it follows that determining the quantities $r_J$ for all $J \subset \{1,\dots,N\}$ allows us to obtain the $r$-th generalized weight enumerator  polynomial of $C$.
It will be convenient to introduce the numbers $B_{j,i}$ defined to be the number of $J \subset \{1, \dots, N\}$ such that $|J|=j$ and rank$(G_J)=i$. In terms of $B_{j,i}$ we can write
\[  B_j(T) = \sum_{i \geq 1} B_{j,i}( T^{K-i}-1). \]
We note that  in the polynomial $B_j(T)$ the quantities $B_{j,K}$ are multiplied by $(T^{K-K}-1)=0$. Therefore, there is no need to determine the quantities $B_{j,K}$. In other words, we may re-define
\beq \label{eq:Bji}  B_j(T) = \sum_{i=1}^{K-1} B_{j,i}( T^{K-i}-1). \eeq

By  \cite[Theorem 5.4]{JP}, we have     $A_0(T)=1$, and for $j >0$:
  \beq \label{eq:AB1}  A_j(T) = \sum_{i=N-j}^{N-d} (-1)^{N+j+i}  \textstyle \tbinom{i}{N-j}  B_i(T).  \eeq

It will be useful to note that :
\beq \label{eq:Bjj} 
B_{j,j}= \textstyle \tbinom{N}{j} - \sum_{i<j} B_{j,i}, \quad \text{ if $j \leq K$}. \eeq

Although, in the expression \eqref{eq:Bji} the quantities $B_{j,K}$ were not required, for later use we note that
\beq  \label{eq:BjK} B_{j,K}= \textstyle \tbinom{N}{j} - \sum_{i<K} B_{j,i}, \quad \text{ if $j > K$.} \eeq

 We now return to the code $C_3$. Here the Veronese embedding $\imath:PG(3,q) \to PG(9,q)$ is given by:
\beq \imath(x_0, \dots, x_3)= (x_0^2, x_0 x_1, x_1^2, x_0 x_2, x_1x_2,x_2^2, x_0x_3,x_1 x_3,x_2 x_3, x_3^2).\eeq
The generator matrix $\bG_3$ of $C_3$ will be denoted as just $\bG$. We need  some more notation and definitions. \begin{enumerate}
\item [--] A \emph{configuration} $S$ is just a set of points of $PG(3,q)$. 
\item [--] The group  $PGL_4(q)$ will be denoted as $G$. 
\item [--] For a homogeneous ideal $I$ of $\bF_q[X_0,X_1,X_2,X_3]$ generated by quadratic forms $f_1, \dots, f_r$, let $V(I)$ denote the intersection of the quadrics $f_1= \dots= f_r=0$ in $PG(3,q)$.\\  For $S \subset PG(3,q)$, let $I_2(S)$ denote the vector space of quadratic forms $f(X_0,X_1,X_2,X_3)$ vanishing on $S$.
\item [--] The \emph{rank} of a  configuration $S$  is defined by any one of the following equivalent numbers: \begin{enumerate} 
\item one more than the dimension of the linear subspace of $PG(3,q)$ spanned by $\imath(S)$. 
\item  the rank of the submatrix of $\bG$ on the columns representing the points of $\imath(S)$. 
\item $10 -\text{dim}(I_2(S))$
\end{enumerate}
\end{enumerate}

\begin{definition} A configuration $S \subset PG(3,q)$ of rank $r$ is a maximal rank $r$ configuration -- in short a $\mC_r$-- if there is no rank $r$ configuration properly containing $S$. 
Equivalently, $S$ is  a $\mC_r$ if  $V(I_2(S))=S$ and $\text{dim}(I_2(S))=10-r$.
\end{definition}
For example, the line $\mL$ given by $x_2=x_3=0$ has rank $3$ because $\imath(\mL)$ is the conic $z_1^2=z_0z_2$ in the plane 
$z_3=\dots=z_9=0$ and the span of this conic is the whole plane, or in other words the vector space $I_2(\mL)$ has basis $X_3^2,X_3X_2,X_3X_1,X_3X_0,X_2^2,X_2X_1,X_2X_0$.
Clearly $V(I_2(\mL))$ is the points of intersection of the hyperplanes $X_3=0$ and $X_2=0$  which is  the line $\mL$. Thus,  $\mL$ is a $\mC_3$. 
Similarly, the plane $\mH$  given by $x_3=0$,  has rank $6$ because $\imath(\mH)$ is a  Veronese surface in the $5$-dimensional linear subspace  $PG(5,q) \subset PG(9,q)$  given by $z_6=\dots=z_9=0$, and the linear span of this surface is  $PG(5,q)$. Here $I_2(\mH)$ has basis $X_3^2,X_3X_2,X_3X_1,X_3X_0$ and clearly $V(I_2(\mH))$ is just the hyperplane $X_3=0$ which is $\mH$. Thus, $\mH$ is a $\mC_6$. 

 We note that there is a monomorphism $\rho: G \hookrightarrow PGL_{10}(q)$ (see \cite[Theorem 4.15]{GGG}) such that for each $g \in G $ and $v \in PG(3,q)$ we have 
\[ \imath(g \cdot v) = \rho(g) \imath(v), \]
and the image  $\mG(\mV_3) \subset PGL_{10}(q)$ of this homomorphism is  the automorphism group of $\mV_3$ in $PGL_{10}(q)$.  Therefore, given a line $\mL$ (resp.  plane $\mH$) of $PG(3,q)$ we may assume that up to $G$-equivalence, $\mL$ is given by $x_2=x_3=0$ (resp. $\mH$ is given by $x_3=0$).  Thus, each  line of $PG(3,q)$ is a $\mC_3$ and each  plane of $PG(3,q)$ is a $\mC_6$.\\

\begin{lemma} \label{maxl_uniq} Given a rank $r$ configuration $S$, there is a unique maximal rank $r$ configuration $\bar S \supset S$.  We have $\bar S = V(I_2(S))$. \end{lemma}
\bep Let $S' = V(I_2(S))$. We must show that i) $S'$ has rank $r$, and ii) any rank $r$ configuration $T$ containing $S$, is contained in $S'$. \\
As for i), it suffices to show $I_2(S') = I_2(S)$.   Since $S' \supset S$, we have $I_2(S') \subset I_2(S)$. On the other hand,  if $f \in I_2(S)$ then by definition of $S'$, it follows that $f$ vanishes on $S'$, and hence $f \in I_2(S')$, i.e. $I_2(S) \subset I_2(S')$.  Thus, $I_2(S')=I_2(S)$.\\
As for ii), if $T \supset S$ has rank $r$, then the  fact that $\text{rank}(T)=\text{rank}(S)=r$ implies $I_2(S)=I_2(T)$, and hence $T \subset V(I_2(T))=V(I_2(S))=S'$. 
 We define $\bar S =V(I_2(S)) $.
\eep

The next lemma allows us to build $\mC_r$'s  by adding points to each $\mC_{r-1}$.

\begin{lemma} \label{max_lem}
If $S$ is a $\mC_r$, and $T\subset S$ then $\bar T \subset S$. In particular, $S$ contains a $\mC_{r-1}$ configuration $T$.
\end{lemma}
\bep  Since $T \subset S$, we get $I_2(S) \subset I_2(T)$, and hence $ \bar T = V(I_2(T)) \subset V(I_2(S))=S$.

We delete   points of $S$ one at a time until the rank drops by one. If $S'$ is  the resulting rank $(r-1)$ configuration, then  $T=\overline{S'}$ is a $\mC_{r-1}$ configuration contained in $S$. \eep

In order to determine the rank $r$ configurations of a given size $j$, it suffices to look for size $j$ and rank $r$ subsets of each $\mC_r$. There are two steps in our approach. \begin{enumerate}
    \item[Step 1 $\rangle$] Determine the set $\fC_r$ of all $\mC_r$'s. i.e., the collection of all maximal rank $r$ configurations.  We recursively  construct all $\mC_r$ configurations by adding points to $\mC_{r-1}$ configurations.  We will also obtain a decomposition,  \[ \fC_r = \cup_{\alpha \in \{a, b, c, \dots\}} \M_{r \alpha}.\]
    Each of the parts $\M_{r \alpha}$ will be  a class of maximal configurations consisting of  a union of $G$-orbits of $\mC_r$'s. We will also determine the size of each class $\M_{r \alpha}$.\\
    \item[Step 2 $\rangle$ ] 
    For each class $M_{r \alpha}$ and each integer $j >r$, we will determine 
$B_{j,r}(\M_{r \alpha})$, which is  the number of rank $r$ configurations of size $j$, contained in a configuration $T$ representing the class $\M_{r \alpha}$. \\
\end{enumerate}

From Step 1 and Step 2,  we  calculate $B_{j, r}$ for $j> r$ by:
    \beq \label{eq:B_j_r} B_{j, r} = \sum_{ \alpha \in \{a, b, c, \dots\} } | \M_{r \alpha}| \cdot B_{j,r}(\M_{r \alpha}), \qquad j>r. \eeq
    The quantity $B_{r,r}$ is determined by the identity 
    \[ B_{r,r} = \tbinom{N}{r} - \sum_{i<r} B_{r, i}. \]
We also give a systematic procedure to 
determine the quantities $B_{j,r}(\M_{r \alpha})$ recursively with respect to $r$: 
\begin{lemma}
Let $T$ be a $\M_{r \alpha}$ configuration.  For each $i <r$, let  
$\fC_i(T)$ denote the collection of all $\mC_i$ configurations that are contained in $T$, and that have at least $(r+1)$ points. We can write 
$ \fC_i(T) = \cup_{\beta \in \{a,b,c, \dots\}} \M_{i \beta}(T)$  where $\M_{i \beta}(T)$ denotes the class of configurations $\M_{i \beta}$ that are   contained in $T$ and have at least $r+1$ points. 
   We have
    \beq \label{eq:B_j_r_alpha} B_{j, r}(\M_{r \alpha}) = \tbinom{|T|}{j} - \sum_{i <r}  \sum_{ \M_{i \beta}(T) \in \fC_i(T)}    |\M_{i \beta}(T)|  \cdot  B_{j, i}(\M_{i \beta}(T)). \eeq
\end{lemma}
\bep In order to obtain $B_{j, r}(\M_{r \alpha})$, we must subtract from $\tbinom{|T|}{j}$, the number of $j$ point configurations $S\subset T$ having rank 
at most $r-1$. If $S\subset T$ is a $j$ point configuration of rank $i<r$, then by Lemma \ref{max_lem} and Lemma \ref{maxl_uniq}, $\bar S$ is  the unique $\M_{i \beta}(T)$ configuration containing $S$. Thus, the second term in the right hand  side of \eqref{eq:B_j_r_alpha} exactly counts all the $j$-point configurations in $T$ that have rank strictly less than $r$.

\eep

In   \S  \ref{r_leq_5}, we study configurations of rank at most $5$ (for a general finite field $\bF_q$), and in \S \ref{r6}, \S \ref{r7}, \S \ref{r8}  and \S \ref{r9},  we study configurations of rank $6$, $7$, $8$ and $9$ respectively over the field $\bF_3$.  Each of the sections \S \ref{r_leq_5} -- \S \ref{r9} is divided into two subsections dedicated to Step 1 and Step 2 of the above method.

\section{Configurations of rank at most $5$ for a general finite field $\bF_q$} \label{r_leq_5}

\subsection{Classification of $\mC_r$, for $r \leq 5$}\hfill \\
The only configuration of rank $1$  is a single point of $PG(3,q)$, and   the only configuration of rank $2$  is a set of two distinct  points of $PG(3,q)$. We recall that maximal rank $r$ configurations can be built from maximal rank $(r-1)$ configurations by adding points. There are two ways to add a point $P_3$ to the $\mC_2$ given by $\{P_0, P_1\}$: if $P_2$ lies on the line $L = \overline{P_0, P_1}$, then by Lemma \ref{maxl_uniq}, the unique $\mC_3$ containing this configuration of $3$ points is the line $L$. All such configurations form one orbit denoted $\M_{3a} = G \cdot \overline{P_0, P_1}$ of size $\stirling{4}{2}_q=(1+q^2)(1+q+q^2)$.

If $P_2 \notin L$ then the resulting configuration $S = \{P_0, P_1, P_2\}$ of $3$ points in general position in a plane $\Pi$ is also a $\mC_3$: indeed we may assume $S = \{e_0, e_1, e_2\}$ so that $I_2(S)$ has basis 
\[ \{X_0X_1,X_0X_2,X_1X_2,X_0X_3,X_1X_3,X_2X_3,X_3^2\}. \] 
It is easily checked that $V(I_2(S)) = S$. All such $\mC_3$'s form a single $G$-orbit denoted $\M_{3b} = G \cdot \{e_0, e_1, e_2\}$, and of size $q^3(1+q)^2(1+q^2) (1+q+q^2)/6$.  In summary:\\

 The $G$-orbits in $\fC_r$ for $r \in \{1,2,3\}$ are as follows:
\begin{itemize}
 \item []There is one orbit of $\fC_1$ given by $\M_{1a}$  of size $(1+q)(1+q^2)$. This orbit consists of points of $PG(3,q)$.
 \item []There is one orbit of $\fC_2$ given by $\M_{2a}$ of size $\tbinom{(1+q)(1+q^2)}{2}$. This orbit consists of unordered pairs of points  of $PG(3,q)$.
\item []There are $2$ orbits  in  $\fC_3$:
	\begin{enumerate}
		\item [$\M_{3a}$]   of size $|G(1,3)|$. This orbit consists of the lines of $PG(3,q)$.
		
	\item [$\M_{3b}$]  of size $q^3(1+q)^2(1+q^2) (1+q+q^2)/6$. This orbit consists of unordered triples of non-collinear points of $PG(3,q)$.
	\end{enumerate}
\end{itemize}

\subsubsection{Classification of  $\mC_4$'s}\hfill \\
We now turn to determining the set $\fC_4$.  Let $S$ be a $\mC_4$  containing 
$\M_{3a}$ configuration represented by $\overline{e_0, e_1}$. We may assume $S$ contains $T= \{e_2\} \cup \overline{e_0, e_1}$. Here  $I_2(T)$  has basis \[ \{X_0X_2,X_1X_2,X_0X_3,X_1X_3,X_2X_3,X_3^2\},\] and $V(I_2(T))$ is clearly equal to $T$.  Thus, $T$ is a $\mC_4$. The set of all such configurations forms a single orbit $\M_{4a} = G \cdot T$ whose elements consist of a pair $(L,P)$ where $L$ is a line not containing the point $P$. Since there are $\stirling{4}{2}_q $ ways to pick a line of $PG(3,q)$ and $(q^2+q^3)$ ways to pick a point not on this line, we see that the size of $\M_{4a}$ is $\stirling{4}{2}_q \cdot (q^2+q^3)$ which simplifies to $q^2(1+q)(1+q^2)(1+q+q^2)$.\\

Next, we consider $S$ which is a $\mC_4$  containing a $\M_{3b}$ configuration represented by $\{e_0, e_1,e_2\}$. If $S$ contains a point $P$ which lies on one of the $3$ lines $\overline{e_i,e_j}$, then by maximality, $S$ contains a $\M_{3a}$ configuration, which case has already been treated. We now assume  $S$ contains no other point on any line $\overline{e_i,e_j}$, $0\leq i<j\leq 2$. If $S$ contains another point $P$  of  the plane $\Pi$ spanned by $\{e_0, e_1,e_2\}$ then we may assume $S$ contains $T = \{e_0, e_1,e_2,e_0+e_1+e_2\}$ which has rank $4$. Here $I_2(T)$ has basis 
\[ \{ X_3^2,X_3X_2,X_3X_1,X_3X_0, X_0X_1-X_0X_2,X_0X_1- X_1X_2\}, \] 
and clearly $V(I_2(T))$ equals $T$. Thus, $T$ is a $\mC_4$. All such configurations of $4$  points in general position in a plane of $PG(3,q)$ form a single orbit $\M_{4c} = G \cdot T$, because given a such a $4$-tuple $(P_1, \dots, P_4)$, there are $(q-1)q^3$ elements $g \in PGL_4(q)$ with the property that $g\cdot (e_0, e_1,e_2,e_0+e_1+e_2) = (P_1, P_2,P_3,P_4)$. Thus, 
\[\M_{4c} = \tfrac{1}{4!} \tfrac{|PGL_4(q)|}{q^3(q-1)}=q^3(q^4-1)(q^3-1)(q+1)/24.\]

  On the other hand, if $S$ contains a point $P$ not lying on the plane generated by $\{e_0, e_1,e_2\}$ then we may assume $S$ contains $T =  \{e_0, e_1,e_2,e_3\}$. Here $I_2(T)$ has basis \[ \{X_0X_1,X_0X_2,X_1X_2,X_0X_3,X_1X_3,X_2X_3\},\]
   and clearly $V(I_2(T))$ equals $T$. Thus, $T$ is a $\mC_4$. Given $4$ ordered  points $(P_1, \dots, P_4)$ in general position in $PG(3,q)$, there are $(q-1)^3$ elements $g \in PGL_4(q)$ with the property that $G \cdot (e_0, e_1,e_2,e_3)=(P_1, \dots, P_4)$. Thus, all such configurations form a single orbit  which we denote $\M_{4b}= G \cdot \{e_0, e_1,e_2,e_3\}$ of size  
   \[ |\M_{4b}| = |PGL_4(q)|/ (24 (q-1)^3) =q^6 (1+q)^2 (1+q^2) (1+q+q^2)/24.\]
 We summarize this:\\

 The $G$-orbits of $\fC_4$ are:
	\begin{enumerate}
		 \item [$\M_{4a}$]  of size $q^2(1+q)(1+q^2)(1+q+q^2)$. This orbit consists of pairs $(L,P)$ of a line $L$ and a point $P$ of $PG(3,q)$ such that $P \notin L$.
     \item [$\M_{4b}$]  of size  $q^6(1+q+q^2+q^3)(1+q+q^2)(q+1)/24$. This orbit  consists of  unordered $4$-tuples of  points in general position in $PG(3,q)$.
     \item [$\M_{4c}$] of size   $q^3(q^4-1)(q^3-1)(q+1)/24$. This orbit  consists of  unordered $4$-tuples of  points in general position in some plane of $PG(3,q)$.
	\end{enumerate}

 \subsubsection{Classification of  $\mC_5$'s}\label{max_r_5}\hfill \\
Next, we determine the set $\fC_5$.   Let $S$ be a $\mC_5$  containing a $\M_{4a}$ configuration  represented by $\{e_2\} \cup  \overline{e_0,e_1}$. If $S$ contains another  point $P$ of the plane $\Pi$ spanned by $\{e_0, e_1, e_2\}$, then the line $L = \overline{P ,e_2}$ meets the line $\overline{e_0,e_1}$ in a third point $Q$. The configuration $\{e_2, P ,Q\}$ is contained in the $\M_{3a}$ configuration $L$, and hence  by Lemma \ref{max_lem},  $S$ must contain $L$. Thus, $S$ contains the union of the lines $L$ and $\overline{e_0,e_1}$ which we denote $T$. We may take $Q=e_0$. Here $I_2(T)$ has basis 
\[\{X_1X_2,X_0X_3,X_1X_3,X_2X_3,X_3^2\},\] 
and clearly $V(I_2(T))$ equals $T$. Thus, $T$ is a $\mC_5$. All such configurations consisting of $2$ intersecting lines in $PG(3,q)$ form a single orbit $G \cdot T$  which we denote $\M_{5a}$. Since $\M_{5a}$ consists of a pair of lines in a plane, we see that  the size of $\M_{5a}$ is:
\[ |\M_{5a}|=(1+q+q^2+q^3) \cdot \tbinom{1+q+q^2}{2}=q(1+q)^2(1+q^2)(1+q+q^2)/2.\]

If $S$ contains no other  point $P$ of the plane $\Pi$ spanned by $\{e_0, e_1, e_2\}$, then we may assume $S$ contains $T = \{e_2, e_3\} \cup \overline{e_0,e_1}$. Here $I_2(T)$ has basis 
\[\{ X_0X_2,  X_0X_3, X_1X_2, X_1X_3, X_3X_2\}, \] 
and clearly $V(I_2(T))$ equals $T$. Thus, $T$ is a $\mC_5$. The set of all configurations consisting of a line $L$ together with $2$ points on a line $L'$ skew to $L$ form the class $\M_{5b}$ of size
 \[|M_{4a}| \cdot q^3/2= q^5(1+q)(1+q^2)(1+q+q^2)/2.\]

Next, let $S$ be a $\mC_5$  containing a $\M_{4c}$ configuration  represented by $S_0=\{e_0, e_1,e_2,e_0+e_1+e_2\}$.  Suppose  $S$ contains another  point $P$ of the plane $\Pi$ spanned by $\{e_0, e_1, e_2\}$. If $P$ lies on one of the $6$ lines between the $4$ points of $S_0$, then by maximality, the whole line is contained in $S$, and hence $S$ contains a $\M_{4a}$ configuration, which case has already been treated. If $P$ does not lie on any of the $6$ lines of $S_0$ 
(this case is possible only if $q>3$), then $S_0 \cup \{P\}$ 
are $5$ points in general position in the plane $\Pi$ (i.e., no $3$ of the points are collinear). For  $q \geq 4$, given any $5$ points in general position of $PG(2,q)$, there exists a unique non-degenerate conic passing through these points. Moreover, all the non-degenerate conics of $PG(2,q)$ are $PGL_3(q)$ equivalent. Let $T$ be the conic  $X_1^2-X_0X_2$  in the plane $X_3=0$. Here  $I_2(T)$ has basis 
\[ \{ X_3^2,X_3X_2,X_3X_1,X_3X_0, X_1^2-X_0X_2\}, \] 
 and clearly $V(I_2(T))$ equals $T$. Thus, $T$ is a $\mC_5$. Let $\M_{5e}$ denote this class of $\mC_5$'s consisting of a non-degenerate conic in a plane of $PG(3,q)$. 
   Since there are $|PGL_3(q)|/|PGL_2(q)|=(q^5-q^2)$ non-degenerate conics in  a plane, and there are 
$(1+q)(1+q^2)$ planes in $PG(3,q)$, we see that 
   \[ |\M_{5e}| =(1+q)(1+q^2) \cdot |PGL_3(q)|/|PGL_2(q)| = (q^5-q^2)(1+q^2)(1+q).\]

We now assume $S$ is a  $\mC_5$  containing a $\M_{4c}$ configuration  represented by $\{e_0, e_1,e_2,e_0+e_1+e_2\}$, and $S$ contains no other  point $P$ of the plane $\Pi$ spanned by $e_0, e_1, e_2$. We may assume $S$ contains $T= \{e_3,e_0, e_1,e_2,e_0+e_1+e_2\}$. Here $I_2(T)$ has basis \[ \{ X_3X_2,X_3X_1,X_3X_0, X_0X_1-X_0X_2,X_0X_1- X_1X_2\}, \]
  and clearly $V(I_2(T))$ equals $T$. Thus, $T$ is a $\mC_5$.
 Let $\M_{5c}$ denote this class of $\mC_5$'s consisting of $5$ points $4$ of which are in general position in a plane (there are $| \M_{4c}|$ ways to pick these), and the fifth point lies outside this plane (there are $q^3$ ways to pick this point),  Since there are $| \M_{4c}|$ ways to pick the $4$ points, and $q^3$ ways to pick the fifth point outside the plane,  we see that
\[ |\M_{5c}|= |\M_{4c}| \cdot q^3= q^6(q^4-1)(q^3-1)(q+1)/24.\]

It remains to consider $S$ which is a $\mC_5$ containing a $\M_{4b}$ configuration represented by 
$\{e_0, e_1,e_2,e_3\}$. If $S$ contains a point $P$ which lies on one of the $6$ lines 
$\overline{e_i,e_j}$, say $\overline{e_0,e_1}$, then by maximality, $S$ contains $\M_{4a}$ which case has already been treated. If $P$ does not lie on any of these $6$ lines, but it  lies on one of the $4$ planes formed by $\{e_0, e_1,e_2,e_3\}\setminus \{e_i\}$, $i=0,1,2,3$,  $S$ will be a $\M_{5c}$ configuration. Therefore, we assume any $P \in S \setminus\{e_0, e_1,e_2,e_3\}$  has the property that the points of $T=\{e_0, e_1,e_2,e_3,P\}$  are in general position. Up to $G$-equivalence, we may assume $P = e_0+e_1+e_2+e_3$.
Here $I_2(T)$ has basis 
\[ \{X_0(X_2-X_1), X_1(X_2-X_0), X_0(X_3-X_1), X_1(X_3-X_0), X_2(X_3-X_0)\}, \]
 and $V(I_2(T))=T$. Thus, $T$ is a $\mC_5$.  Since $5$-tuples $(P_1, \dots,P_5)$ of points in general position in $PG(3,q)$ correspond bijectively with $PGL_4(q)$, we see that the size of the class $\M_{5d}$ consisting of unordered $5$-tuples of points in general position in $PG(3,q)$ is 
   \[ |\M_{5d}|=|PGL_4(q)|/5! = q^6(q^4-1)(q^3-1)(q^2-1)/120.\] 
   We summarize this:  \\

 The set $\fC_5$ can be partitioned into the following five classes:
\begin{enumerate}
\item [$\M_{5a}$] of size $q(1+q)^2(1+q^2)(1+q+q^2)/2$. This class consists of a pair of intersecting lines of $PG(3,q)$.
\item [$\M_{5b}$] of size $q^5(1+q)(1+q^2)(1+q+q^2)/2$. This class consists of a line $L$ and two points on a line $L'$ skew to $L$.
\item [$\M_{5c}$] of size $q^6(q^4-1)(q^3-1)(q+1)/24 $.  This class consists of four  points in general position in a plane, together with a point outside the plane.
\item [$\M_{5d}$] of size $q^6(q^4-1)(q^3-1)(q^2-1)/120$.  This class consists of
 $5$ points of $PG(3,q)$ in general position.
 \item [$\M_{5e}$] of size $(q^5-q^2)(1+q^2)(1+q)$. This class consists of a non-degenerate conic in a plane of $PG(3,q)$.
\end{enumerate}

\subsection{Calculations of $B_{j,i}$ for $j>i$}\hfill \\
Since a maximal configuration of rank $1$ and $2$ contains exactly $1$ and $2$ points of $PG(3,q)$ respectively. Therefore, in these cases we can see that $B_{j,i}=0$ for $i \in \{1,2\}$ and  $j>i$.

\subsubsection{$B_{j,3}$ for $j>3$}\hfill \\
Let $T$ be a $\M_{3a}$ configuration, i.e., a line in $PG(3,q)$. Then $T$ has $(1+q)$ points of $PG(3,q)$. Since $j>3$ and any $S \in \fC_i(T)$, $i\leq 2$ has at most $2$ points. By \eqref{eq:B_j_r_alpha}, we have \[B_{j,3}(\M_{3a})=\tbinom{1+q}{j}. \]

A $\M_{3b}$ configuration has $3$ points, so $B_{j,3}(\M_{3b})=0$. Therefore from \eqref{eq:B_j_r}, we have

\begin{equation}\label{BJ3}
    B_{j,3}=|\M_{3a}|\cdot \tbinom{1+q}{j} =(1+q^2)(1+q+q^2)\cdot \tbinom{1+q}{j}.
\end{equation}

\subsubsection{$B_{j,4}$ for $j>4$}\hfill \\
An $\M_{4a}$ configuration $T$ represented by a pair $(L,P)$ of a line $L$ and a point $P$ of $PG(3,q)$,  has size  $(q+2)$ points. Thus:
\[ \fC_1(T)=\fC_2(T)= \varnothing, \; \fC_3(T)= \M_{3a}(T). \]
Since there is only one line in $T$, we have  $|\M_{3a}(T)|=1$. By \eqref{eq:B_j_r_alpha}, we have, \[B_{j,4}(\M_{4a})=\tbinom{q+2}{j}-\tbinom{1+q}{j}.\]

Since the $\M_{4c}$ and $\M_{4b}$ configurations have exactly $4$ points of $PG(3,q)$ and $j>4$, we see that $B_{j,4}(\M_{4b})=B_{j,4}(\M_{4c})=0$. Therefore for $j>4$, by \eqref{eq:B_j_r}, we have

\begin{equation}\label{BJ4}
    B_{j,4}= q^2(1+q)(1+q^2)(1+q+q^2) \cdot (\tbinom{q+2}{j}-\tbinom{1+q}{j}).
\end{equation}
  
\subsubsection{$B_{j,5}$ for $j>5$}\hfill \\
A $\M_{5a}$ configuration $T$ represented by a pair of intersecting lines of $PG(3,q)$ contains $(2q+1)$ points. Thus,
\[ \fC_1(T)=\fC_2(T)=\varnothing, \quad  \fC_3(T)= \M_{3a}(T), \quad 
\fC_4(T)= \M_{4a}(T). \]
We also note that $|\M_{3a}(T)|=2$ and $|\M_{4a}(T)|=2q$.
By \eqref{eq:B_j_r_alpha}, we have, \[B_{j,5}(\M_{5a})=\tbinom{2q+1}{j}-2q \cdot (\tbinom{q+2}{j}-\tbinom{1+q}{j})-2 \tbinom{1+q}{j}=\tbinom{2q+1}{j}-2q \cdot \tbinom{q+2}{j}+2(q-1)\cdot \tbinom{1+q}{j}.\]

\subsubsection*{}
A $\M_{5b}$ configuration $T$ represented by a line $L$ and two points $\{P_1,P_2\}$ on a line $L'$ skew to $L$ contains $(q+3)$ points of $PG(3,q)$. Again,
\[ \fC_1(T)=\fC_2(T)=\varnothing, \quad  \fC_3(T)= \M_{3a}(T), \quad 
\fC_4(T)= \M_{4a}(T). \]
We also note that $|\M_{3a}(T)|=1$ and $|\M_{4a}(T)|=2$.  By \eqref{eq:B_j_r_alpha}, we have
\[B_{j,5}(\M_{5b})=\tbinom{q+3}{j}-2 \cdot (\tbinom{2+q}{j}-\tbinom{1+q}{j})-\tbinom{1+q}{j}=\tbinom{q+3}{j}-2 \cdot \tbinom{2+q}{j}+\tbinom{1+q}{j}.\]

\subsubsection*{}
A $\M_{5c}$ configuration $T$ represented by the $(q+1)$ points of a non-degenerate conic in a plane of $PG(3,q)$. Here 
\[ \fC_1(T)=\fC_2(T)=\fC_3(T)=\fC_4(T)=\varnothing, \]
because $T$ does not contain a line which rules out $\M_{3a}, \M_{4a} $ and configurations in $\fC_i(T)$ must have at least $6$ points which rules out $\M_{3b}, \M_{4b}, \M_{4c}$.
 By \eqref{eq:B_j_r_alpha}, we have 
 \[B_{j,5}(\M_{5c})=\tbinom{1+q}{j}.\]

Since each of the $\M_{5d}$ or $\M_{5e}$ configurations contains $5$ points $PG(3,3)$, we have for $j>5$,
\[ B_{j,5}(\M_{5d})=B_{j,5}(\M_{5e})=0.\]
Therefore for $j>5$, we have
\begin{equation}\label{BJ5}
\begin{split}
 B_{j,5}=&\tfrac{1}{2}q(1+q)^2(1+q^2)(1+q+q^2)\cdot \left( \tbinom{2q+1}{j}-2q \cdot \tbinom{q+2}{j}+2(q-1)\cdot \tbinom{1+q}{j}\right)\\
 &{}+\tfrac{1}{2}q^5(1+q)(1+q^2)(1+q+q^2)\cdot (\tbinom{q+3}{j}-2 \cdot \tbinom{2+q}{j}+\tbinom{1+q}{j})\\
 &{}+(q^5-q^2)(1+q^2)(1+q)\cdot \tbinom{1+q}{j}.   
\end{split}
\end{equation}

\section{Configurations of rank $6$ }
\label{r6}
For the rest of this paper, we assume $q=3$. The size of the group $G = PGL_4(q)$ is
\[|G|=12130560=   2^8 \cdot 3^6 \cdot 5 \cdot 13.\]
Since an irreducible conic in a projective plane has only $4$ points when $q=3$. We see that the class of configuration $\M_{5e}$ in $\fC_5$ does not occur when $q=3$. Therefore, the set $\fC_5$ is partitioned into the following four classes when $q=3$:

\begin{enumerate}
\item [$\M_{5a}$] of size $2^4\cdot 3\cdot5\cdot13$.
This class consists of a pair of intersecting lines of $PG(3,3)$.

\item [$\M_{5b}$] of size $2^2 \cdot  3^5 \cdot 5  \cdot 13$  This class consists of a line $L$ and two points on a line $L'$ skew to $L$.

\item [$\M_{5c}$] of size $2^4 \cdot  3^5 \cdot 5  \cdot 13$.   This class consists of four  points in general position in a plane, together with a point outside the plane.
\item [$\M_{5d}$] of size $2^5 \cdot  3^5  \cdot 13$. This class consists of $5$ points of $PG(3,3)$ in general position.
\end{enumerate}

As before, we build the $\mC_6$'s by adding points to $\mC_5$'s. We carry this out systematically by running through the classes in $\fC_5$.

\begin{lemma}\label{r-6}
The set $\fC_6$ can be partitioned into the following 9 classes: 
 \begin{enumerate}
\item [$\M_{6a}$] of size $2^3 \cdot 5$. It  consists of planes of $PG(3,3)$.
\item [$\M_{6b}$]  of size  $2^4 \cdot  3^4 \cdot 5  \cdot 13$.
It  consists of configurations $(\{L_1, L_2\},\{P\})$ where $\{L_1, L_2\}$ is  a pair of intersecting lines, and $P$ does not lie on the plane generated by the two lines.
\item [$\M_{6c}$]   of size  $3^4 \cdot 5  \cdot 13$.
It  consists of a pair of skew lines in $PG(3,3)$.
\item [$\M_{6d}$]   of size $2^6 \cdot  3^4 \cdot 5  \cdot 13$. It consists of configurations $(L, \{P_1, \dots, P_4\})$ where $\{P_1, \dots, P_4\}$  are $4$   points in general position in a plane $\Pi$ and the line $L$ intersects $\Pi$ in one of these $4$  points.
\item [$\M_{6e}$]   of size $2^4 \cdot  3^5 \cdot 5  \cdot 13$.
It consists of configurations $\{P_1, \dots, P_6\}$ where the $3$ lines $L_i = \overline{P_{2i-1} ,P_{2i}}$ are non-coplanar lines meeting at a single point different from $\{P_1, \dots, P_6\}$.
\item [$\M_{6f}$]   of size $2^4 \cdot 3^6 \cdot 5 \cdot 13$,
The elements of $\M_{6f}$ are configurations $\{P_1, \dots, P_6\}$ where   $P_1, P_2$ are on a line $L$,  and $P_3, P_4$ are on a line $L'$,  and $P_5, P_6$ are on a line $L''$ with $\{L', L''\}$ being a pair of skew lines which meet $L$ at distinct points $P, Q \notin \{P_1, P_2\}$.

\end{enumerate}
\end{lemma}

\begin{proof}
Let $S$ be a $\fC_6$ configuration.
\subsubsection*{{\bf Classes $\M_{6a}, \M_{6b}$}} \hfill \\ We begin with $\mC_6$  configurations $S$ containing a $\M_{5a}$ configuration represented by  the pair of lines $X_1X_2=0$ in the plane $X_3=0$. If $S$ contains another point $P$ of the plane $X_3=0$, then by maximality of $\M_{5a}$, the configuration $\{P\} \cup \M_{5a}$ has rank $6$. We have already shown that a plane of $PG(3,3)$ is $\mC_6$, therefore, $S$ is represented by the plane $X_3=0$. Let $\M_{6a}$ denote the class of configurations consisting of planes of $PG(3,3)$. The size of this class is $40$.

If $S$ contains no other point of the plane $X_3=0$, then we may assume $S$ contains $T = \overline{e_0,e_1}\cup \overline{e_0,e_2}\cup \{e_3\}$.  Here $I_2(T)$ has basis $\{X_1X_2,X_0X_3,X_1X_3,X_2X_3\}$ and clearly $V(I_2(T))$ equals $T$. Thus, $T$ is a $\mC_6$.
Let $\M_{6b}$ denote the class of configurations consisting of a pair of intersecting lines (which can be chosen in $|\M_{5a}|$ ways) together with a point not on the plane generated by the pair of lines (which can be chosen in $3^3$ ways).  The size of  
this class is
\[  |\M_{6b}|= |\M_{5a}|\cdot 3^3 = 3^4(1+3)(1+3+3^2)(1+3+3^2+3^3)/2=84240 .\]

\subsubsection*{{\bf Classes $\M_{6c}, \M_{6d}$}}\hfill \\
Next, we consider $\mC_6$  configurations $S$ which contain the  $\M_{5b}$ configuration represented by  $S_0 =  \overline{e_0,e_1} \cup \{e_2,e_3\}$. Let $L = \overline{e_0,e_1}$ and let $L'=\overline{e_2,e_3}$.  If $P \in S$ lies on the line $L'=\overline{e_2,e_3}$, then by maximality, $S_0 \cup \{P\}$ has rank $6$, and is contained in the configuration $T=L \cup L' = \overline{e_0,e_1} \cup \overline{e_2,e_3}$.  Here $I_2(T)$ has basis $\{X_0X_2,X_0X_3,X_1X_2,X_1X_3\}$ and clearly $V(I_2(T))$ equals $T$. Thus, $T$ is a $\mC_6$. Let 
$\M_{6c}$ denote this class consisting of a pair of skew lines. This class has size
\[  |\M_{6c}|= 130\cdot 3^4/2=5265.\]

We now assume  $S \supset S_0$ does not contain any other point of $L \cup L'$. Given a point $P$ of $S \setminus S_0$, if one of the lines $\overline{P, e_2}$ or $\overline{P ,e_3}$ meets the line $L = \overline{e_0,e_1}$, then by maximality, $S$ contains a $\M_{5a}$ configuration, which has already been treated. Thus, the three points $P, e_2, e_3$ are non-collinear and $L$ does not intersect any of the $3$ lines generated by $P, e_2,e_3$. Up to $G$-equivalence, we may assume $T=\{P\} \cup S_0$ is represented by  $\{e_0, e_1,e_2\} \cup \overline{e_3,e_0+e_1+e_2}$.  Here $I_2(T)$ has basis 
 \[ \{X_3(X_2-X_0),  X_3(X_1-X_0),   X_0(X_1-X_2),   X_1(X_0- X_2) \}, \] and we verify that  $V(I_2(T))$ equals $T$. Thus, $T$ is a $\mC_6$. Let $\M_{6d}$ be the class  consisting of $4$  points in general position in a plane $\Pi$  (which can be chosen in $|\M_{4c}|$ ways), together with a line 
 which intersects $\Pi$ in one of these $4$  points (which can be chosen in $4\cdot 3^2$ ways). 
 Thus, the size of this class is
 \[ |\M_{6d}| = 4 \cdot 3^2 \cdot |\M_{4c}|  = 3^5(3^4-1)(3^3-1)(3+1)/6=336960.\]

\subsubsection*{{\bf Classes $\M_{6e}, \M_{6f}$}} \hfill \\
Next we consider $\mC_6$  configurations $S$ which contain the  $\M_{5c}$ configuration $S_0= 
\{e_3,e_0, e_1,e_2,e_0+e_1+e_2\}$. If $S$ contains a point $P$ of the plane $X_3=0$, then by maximality, $S$ contains a $\M_{5a}$, which case has already been treated. So we assume $S \setminus S_0$ contains  no other point of the plane $X_3=0$. For $P \in S \setminus S_0$, let  $T = \{P, e_3,e_0, e_1,e_2,e_0+e_1+e_2\}$.

Let  $Q$ be the intersection point of the  line $\overline{e_3, P}$ with the plane $X_3=0$.
Consider the $6$ lines $L_1, \dots, L_6$ in the plane $X_3=0$  generated by $\{P_1, \dots, P_4\}=\{e_0, e_1,e_2,e_0+e_1+e_2\}$. Since $q=3$, the union of these $6$ lines is the whole plane $X_3=0$. If a point $R$ lies on $3$ of these lines then $R \in \{P_1, \dots, P_4\}$. There are three points $R_1 = e_0+e_1,R_2=e_1+e_2,R_3=e_2+e_0$ which lie on exactly two of these $6$ lines. There are three possibilities for the point $Q$ above: i) $Q \in \{e_0, e_1,e_2,e_0+e_1+e_2\}$, ii) $Q \in \{R_1,R_2,R_3\}$, iii) $Q$ lies on exactly one of the lines $L_1, \dots, L_6$.  We consider each of these possibilities one by one. \\

i) If $Q\in\{P_1, \dots, P_4\}$,  then by maximality, $S$ contains a $\M_{5b}$ configuration, which case has already been treated.\\

ii)  If $Q \in \{R_1, R_2, R_3\}$ say $Q = e_0+e_1$, then we may assume $P = e_0+e_1+e_3$  (by replacing $e_3$ with $c e_3$ for a suitable scalar $c$). We note that each of the three lines $\overline{Q ,e_0}, \overline{Q ,e_2}, \overline{Q ,e_3}$ through $Q$ contain exactly $2$ points of $T$, namely $\{e_0, e_1\}, \{e_2, e_0+e_1+e_2\}, \{e_3, e_0+e_1+e_3\}$. Thus, $T$ is $G$-equivalent to the configuration $\{e_0, e_0+e_3, e_1, e_1+e_3,e_2, e_2+e_3\}$, with $e_3$ representing the point of concurrence of the $3$ lines $\overline{e_3, e_i}$ for $i=0,1,2$. Here $I_2(T)$ has basis 
\[ \{ X_0X_1, X_1X_2, X_0X_2, X_3(X_0+X_1+X_2-X_3)\}, \] 
 and clearly $V(I_2(T))$ equals $T$. Thus, $T$ is a $\mC_6$. Let $\M_{6e}$ be this class of configurations consisting of $2$ points each on three non-coplanar lines through a point $P$. 
The configuration $\M_{6e}$ is obtained by i) picking  a point $P$ and three non-coplanar lines through $P$, and ii) picking $2$ points on each of these $3$ lines different from $P$.   
The size of this  class is
\[|\M_{6e}|=  40 \cdot  \tfrac{1}{6} (1+3+3^2)(1+3)3^3   \cdot  \tbinom{3}{2}^3 =252720.\]
For later use, we note that  the stabilizer group in $G$ of the $\M_{6e}$ configuration $T$ is group of order $48$ which is the wreath product $Z_2 \wr S_3$ where $S_3$ permutes the three lines through $e_3$ (generated by $\bbsm P&0\\0&1 \besm $ where $P$ is  $3 \times 3$ permutation matrix) and $Z_2$ interchanges the $2$ points on a line, (for example generated by $g_0=\bbsm -1&0&0&0\\0&1&0&0\\0&0&1&0\\-1&0&0&1 \besm$ for the line through $\overline{e_0,e_3}$). In more detail, if $g_0, g_1, g_2$ are involutions interchanging the points $\{ e_i, e_i + e_3\}$ for $i=0,1,2$ respectively, then $g_0, g_1, g_2$ commute with each other and hence generate the group $Z_2 \times Z_2 \times Z_2$. The group $S_3$ above acts on 
 $Z_2 \times Z_2 \times Z_2$ by permutation action on the three factors.  As a corollary, we note that given a point $Q = ae_0+be_1+ce_2+de_3$ with $abc \neq 0$, there exist automorphisms of the $\M_{6e}$ configuration that allow us to independently change the sign of $a, b$ and $c$ (the coordinate $d$ is not necessarily fixed). Thus, up to $G$-equivalence we can assume $Q = e_0+e_1+e_2+d' e_3$ for some $d'\in \bF_3$. 

iii) If $Q$ is on exactly one of the lines $L_1, \dots, L_6$,  say $\overline{e_0,e_1}$ then $Q \notin \{R_1, R_2,R_3\}$ and hence $Q \neq e_0+e_1$. Thus, $P = e_0+ \lambda e_1+e_3$ for some $\lambda \in \bF_3 \setminus\{0,1\}$, i.e. $\lambda=-1$.  Thus,
 \[ T = \{e_0,e_1,e_2,e_0+e_1+e_2, e_3, e_3+e_0-e_1\}.\]
 Here $I_2(T)$ has basis 
 \[ \{ X_3X_2, X_3(X_1+ X_0), X_0(X_1-X_2+ X_3), X_2(X_0-X_1)\},\] 
 and it is readily checked that  $V(I_2(T))$ equals $T$. Thus, $T$ is a $\mC_6$. We will denote this class configuration by $\M_{6f}$. We note that $6$ points of $T$ generate $14$ planes of which only $2$ of them have $4$ points: these are the planes $X_2=0$ and $X_3=0$. Thus, the class $\M_{6f}$ consists of configurations $G$-equivalent to $T$. An element of $\M_{6f}$ can be constructed by  i) choosing $2$ points $P_1, P_2$  on the line of intersection of two planes $\Pi_1, \Pi_2$ ii) $2$ points $P_3, P_4$ on  the plane $\Pi_1$ and iii) $2$ points $P_5, P_6$ with the property that if $Q_1 = \overline{P_3,P_4} \cap L$ and $Q_2=\overline{P_5,P_6} \cap L$ then the points $P_1, P_2, Q_1, Q_2$ are distinct points of $L$.
 The pair of planes can be chosen in $\tbinom{40}{2}$ ways. The points $P_1, P_2$ can be chosen in $\tbinom{3+1}{2}$ ways, and the points $Q_1, Q_2$ can be chosen in $(3-1)(3-2)$ ways. The points $P_3, P_4$ and $P_5,P_6$ can be chosen in $(3 \tbinom{3}{2})^2$ ways.
Thus, the  class $\M_{6f}$ has size
\[ |\M_{6f}| = \tbinom{40}{2} \cdot \tbinom{3+1}{2}  \cdot (3-1)(3-2) \cdot (3 \tbinom{3}{2})^2 =758160. \]
For later use, we record   the stabilizer group in $G$ of the $\M_{6f}$ configuration $T$.
Consider the following four involutions which preserve $T$:
\[ g_0=\bbsm &1&&\\1&&&\\&&1&\\&&&-1\besm, \,
g_1=\bbsm &1&-1&\\1&&-1&\\&&-1&\\&&&-1
 \besm, \, 
 g_2=\bbsm &1&&1\\1&&&-1\\
&&1&0\\&&&1
 \besm, \,
 g_3=\bbsm -1&&&\\&1&&\\&&&-1\\
&&-1&\besm. \] 
 We note that i) $g_0$ interchanges $\{P_1,P_2\}$, 
ii) $g_1$ interchanges $\{P_1,P_2\}$,  $\{P_3,P_4\}$,  iii) $g_2$  interchanges $\{P_1,P_2\}$,  $\{P_5,P_6\}$, and iv) $g_3$  interchanges $\{P_3,P_5\}$ and $ \{P_4,P_6\}$. 
The matrices $g_1, g_2$ commute and generate the group $Z_2 \times Z_2$. The matrix $g_3$
  and satisfies $g_3g_1g_3=g_2$, $g_3g_2g_1=g_2$. Thus, the group generated by $g_1, g_2, g_3$  is isomorphic to
the dihedral group.
The element $g_0$ commutes with $g_2, g_3, g_4$. Thus, the stabilizer group of $T$ 
 is isomorphic to the direct product $D_4 \times Z_2$
 \[ Z_2 \times D_4 : \langle g_0, g_1,g_2,g_3 \;|\; g_0^2,g_1^2,g_2^2,g_3^2, (g_1g_2)^2,(g_0g_1)^2,(g_0g_2)^2,(g_0g_3)^2,g_3g_1g_3g_2 \rangle. \]

Finally, we come to the case when $S$ contains a $\M_{5d}$  represented by $5$ points $\{P_1, \dots, P_5\}$ in general position in $PG(3,3)$. We may take $\{P_1, \dots, P_5\}=\{e_0, e_1,e_2,e_3, e_0+e_1+e_2+e_3\}$.  If $P$ is another point of $S$, let $T = \{P, P_1, \dots, P_5\}$. Since in $PG(3,3)$ there is no collection of $6$ points in  general position, there are only two possibilities: i)  some $3$ points of $T$ are collinear, or ii)  no $3$ points of $T$ are collinear but some $4$ points of $T$ are coplanar. In case (i), we may assume $P$ lies on the  line $\overline{P_4P_5}$. In this case by maximality, $S$ contains a $\M_{6d}$ configuration, and  
in case (ii),  $S$ contains a $\M_{6f}$ configuration. These cases have already been treated above.

\end{proof}

\subsection{$B_{j,6}$ for $j>6$}
\subsubsection*{}
A $\M_{6a}$ configuration $T$ represented by a plane $\Pi$ contains $13$ points of $PG(3,3)$. Since $j \geq 7$, we have 
\[ \fC_1(T) =\dots=\fC_4(T)=\varnothing, \quad \fC_5(T)=\M_{5a}(T).\]
We note that $|\M_{5a}(T)|=\tbinom{13}{2}=78$. By \eqref{eq:B_j_r_alpha}, we have, \[B_{j,6}(\M_{6a})=\tbinom{13}{j}-78\tbinom{7}{j}.\]

\subsubsection*{}
A $\M_{6b}$ configuration $T$ of  $8$ points of $PG(3,3)$ represented by a pair of intersecting lines $L_1, L_2$ and point $P$ outside the plane $\langle L_1, L_2\rangle$. Here, again:
\[ \fC_1(T) =\dots=\fC_4(T)=\varnothing, \quad \fC_5(T)=\M_{5a}(T),\]
and $|\M_{5a}(T)|=1$. By \eqref{eq:B_j_r_alpha}, we have, \[B_{j,6}(\M_{6b})=\tbinom{8}{j}-\tbinom{7}{j}.\]

\subsubsection*{}
A $\M_{6c}$ configuration $T$ of  $8$ points of $PG(3,3)$ represented by a pair of skew lines $L_1, L_2$. Here $\fC_i(T)=\varnothing$ for all $1 \leq i \leq 5$. By \eqref{eq:B_j_r_alpha}, we have, \[B_{j,6}(\M_{6c})=\tbinom{8}{j}.\]

\subsubsection*{}
A $\M_{6d}$ configuration $T$ represented by $(L, \{P_1, \dots, P_4\})$ where $\{P_1, \dots, P_4\}$  are $4$   points in general position in a plane $\Pi$, and the line $L$ intersects $\Pi$ at $P_4$. The configuration $T$ contains $7$ points of $PG(3,3)$. Here again, $\fC_i(T)=\varnothing$ for all $1 \leq i \leq 5$, and hence by \eqref{eq:B_j_r_alpha}, we have, \[B_{j,6}(\M_{6d})=\tbinom{7}{j}.\]

We have, $B_{j,6}(\M_{6e})=B_{j,6}(\M_{6f})=0$ as each of $\M_{6e}$ or $\M_{6f}$ configurations contain exactly $6$ points of $PG(3,3)$. Therefore, by \eqref{eq:B_j_r}, for $j>6$, we have 

\begin{equation*}
B_{j,6}=40 \cdot (\tbinom{13}{j}-78\tbinom{7}{j})+84240\cdot (\tbinom{8}{j}-\tbinom{7}{j})+5265\cdot \tbinom{8}{j}+336960\cdot \tbinom{7}{j},
\end{equation*}
i.e. 
\begin{equation}\label{BJ6}
B_{j,6}=40 \cdot \tbinom{13}{j}+89505\cdot \tbinom{8}{j}+249600\cdot \tbinom{7}{j}.
\end{equation}

\section{Configurations of rank $7$} \label{r7}
As before, we build the $\mC_7$'s by adding points to $\mC_6$'s. We carry this out systematically by running through the classes in $\fC_6$.
\begin{lemma} \label{list7}
The set $\fC_7$ can be partitioned into the following 7 classes:  
\begin{enumerate}
\item [$\M_{7a}$] of size $2^3 \cdot 3^3 \cdot 5$.
 This class consists of a plane of $PG(3,3)$ and a point outside that plane. 

\item [$\M_{7b}$] of size $2^4 \cdot 3^2 \cdot 5 \cdot 13$. This class consists of $3$ concurrent non coplanar lines of $PG(3,3)$.

\item [$\M_{7c}$] of size 
$ 2^4 \cdot 3^4 \cdot 5 \cdot 13$. This class consists of three lines $L_1,L_2,L_3$ of $PG(3,3)$, where $L_2$ and $L_3$ are skew to each other and the line $L_1$ meets both the lines. 

\item [$\M_{7d}$] of size 
$2^5 \cdot 3^5 \cdot 5 \cdot 13$. This class consists of configurations $(\{L_1,L_2\},\{P_1,P_2\})$, where $\{L_1,L_2\}$ is a pair of intersecting lines and $P_1,P_2$ lie outside the plane $\langle L_1, L_2\rangle$. Also, the line $\overline{P_1,P_2}$ does not meet $L_1$ and $L_2$. 

\item [$\M_{7e}$] of size 
$ 2^5 \cdot 3^6 \cdot 5 \cdot 13$.
 This class consists of configurations $(\{P_1,P_2,P_3,P_4\},L)$, where the points $P_1,P_2,P_3,P_4$ are in general position in $PG(3,3)$ and the line $L$ does not meet any of the $6$ lines $\overline{P_i,P_j}$, where $1\leq i < j \leq 4$. 

\item [$\M_{7f}$] of size $2^6 \cdot 3^5 \cdot 5 \cdot 13$. This class consists of configurations $\{P_1, \dots, P_6,P_7\}$ where the $3$ lines $L_i = \overline{P_{2i-1} ,P_{2i}}$, $i=1,2,3$, are non-coplanar lines meeting at a single point $Q$ different from $\{P_1, \dots, P_6\}$.  Let  $L_4$ be a line through $Q$  such that no three of the four lines $\{L_1, L_2, L_3, L_4\}$ are coplanar.  Of the three points $\{Q_1, Q_2, Q_3\}$ on $L_4$ different from $Q$, there is only one point $Q_1$ which lies on two of the planes spanned by three points of  $\{P_1, \dots, P_6\}$,   whereas the other two points $\{Q_2, Q_3\}$ lie on three such planes. We take $P_7=Q_1$.
\item [$\M_{7g}$] of size $2^2 \cdot 3^5 \cdot 5 \cdot 13$. This class consists of configurations $\{P_1, \dots, P_8\}$ where the four lines $L_i = \overline{P_{2i-1}, P_{2i}}$, $i=1,2,3,4$, meet at a single point $Q$ different from $\{P_1, \dots, P_8\}$ and no three of these four lines  are coplanar.  If we arbitrarily pick $\{P_{2i-1}, P_{2i}\}$ on $L_i$ for $i = 1, \dots, 3$, then the points $\{P_7, P_8\}$ on $L_4$ are the pair of points $\{Q_2, Q_3\}$ in the notation of the description of $\M_{7f}$
configuration above.
\end{enumerate}
\end{lemma}

\begin{proof}
Let $S$ be a $\mC_7$  configuration. 

\subsubsection*{{\bf Class $\M_{7a}$}}\hfill \\
If $S$ contains a $\M_{6a}$ configuration represented by the plane $\Pi$ given by $X_3=0$, then we may assume  $S$ contains $T = \{e_3\} \cup \Pi$.  Here $I_2(T)$ has basis $\{X_3X_0,X_3X_1, X_3X_2\}$ and $V(I_2(T))$ clearly equals $T$. Thus, $T$ is a $\mC_7$ which we denote $\M_{7a}$. The size of $\M_{7a}$  is $3^3 \cdot |\M_{6a}|= 1080$.

\subsubsection*{{\bf Classes $\M_{7b}$, $\M_{7c}$, $\M_{7d}$}}\hfill \\
If $S$ contains a  $\M_{6b}$ configuration represented by $S_0=\overline{e_0,e_1}\cup \overline{e_0,e_2}\cup \{e_3\}$, and $P$ is another point of $S$, there are $4$ possibilities: i) $P$ lies on the plane $\Pi: X_3=0$, ii) $P$ lies outside $\Pi$, and the line $\overline{P ,e_3}$ meets $\Pi$ at $e_0$, iii) $P$ lies outside $\Pi$, and the line $\overline{P ,e_3}$ meets $\Pi$ at a point $Q$ of $S_0$ different from $e_0$, iv) $P$ lies outside $\Pi$, and the line $\overline{P ,e_3}$ meets $\Pi$ at a point $Q$ not lying on $S_0$.

In  case i)  by maximality,  $S$ contains a $\M_{6a}$ configuration which case has been treated above.

In  case ii)   by maximality, $S$ contains the configuration $T$ of three concurrent lines $\overline{e_0,e_1}\cup \overline{e_0,e_2}  \cup \overline{e_0,e_3}$.   Here $I_2(T)$ has basis $\{X_3X_2,X_3X_1, X_1X_2\}$ and $V(I_2(T))$ clearly equals $T$. Thus, $T$ is a $\mC_7$ configuration. We denote the class of configurations $G$-equivalent to $T$ by  $\M_{7b}$. The size of $\M_{7b}$  is $|\M_{6e}|/ \tbinom{3}{2}^3$, i.e. $9360$.\\

In  case iii)   by maximality of $S$, we may assume $S$ contains the configuration $T$ of the three lines $L_1 \cup L_2 \cup L_3$ where $L_1=\overline{e_0,e_1}$, $L_2=\overline{e_0,e_2}$ and $L_3= \overline{e_1,e_3}$. Here $I_2(T)$ has basis $\{X_3X_0,X_3X_2, X_1X_2\}$ and $V(I_2(T))$ clearly equals $T$. Thus, $T$ is a $\mC_7$. We denote the class of configurations $G$-equivalent to $T$ by  $\M_{7c}$. The lines $L_2$ and $L_3$ are skew to each other and the line $L_1$ meets both the lines. We can choose two skew lines $L_2$ and $L_3$ in $PG(3,3)$ in $130\cdot 3^4/2$ ways and pick the line  $L_1$ in $(3+1)^2$ ways. Thus, the size of the class $\M_{7c}$ is  $65\cdot 3^4 \cdot 4^2$, i.e. $84240$.\\

In  case iv)   we may assume $S$ contains the configuration $T$ consisting of the lines $L_1=\overline{e_0,e_1}$ and  $L_2=\overline{e_0,e_2}$ in the plane $X_3=0$,  and points $\{e_3, e_3+e_1+e_2\}$. Here $I_2(T)$ has basis $\{X_1(X_3-X_2),X_2(X_3-X_1), X_3X_0\}$ and $V(I_2(T))$ clearly equals $T$. We denote the class of configurations $G$-equivalent to $T$ by  $\M_{7d}$. The number of ways to choose  $T$ is the number of ways to pick the configuration  $\M_{6b}$ multiplied by the number of ways to pick the point $P$ which is $(3-1)(3^2-3)$ (there are $(3^2-3)$ ways to pick the point $Q$ and $(3-1)$ ways to pick the point $P$ on the line $\overline{e_3,Q}$). Thus,  $|\M_{7d}|=|\M_{6b}|\cdot \frac{3(3-1)^2}{2}$, i.e. $505440$.\\

In case $S$ contains the  $\M_{6c}$ configuration represented $L_1=\overline{e_0,e_1}$  and $L_2=\overline{e_2,e_3}$. Given any other point $P$ of $S$, we may assume $P=e_0+e_1+e_2+e_3$. In this case the line $L_3=\overline{e_0+e_1, e_2+e_3}$ also contains $P$, and hence by maximality, $S$ contains $L_1 \cup L_2 \cup L_3$ which is the configuration $\M_{7c}$  discussed above.

\subsubsection*{{\bf Class $\M_{7e}$}}\hfill\\
Next we consider  the case when $S$ contains the  $\M_{6d}$ configuration represented by 
 $S_0=\{P_1, \dots, P_4\}$  $ \cup L$ where $\{P_1, \dots, P_4\}$ are  in general position in a plane $\Pi$ and the line $L$  intersects $\Pi$ in $P_4$. Let $P$ be another point of $S$ and let $T = S_0 \cup \{P\}$.
If $P \in \Pi$, then by maximality, $S$ contains a $\M_{6b}$ configuration, which cases have been treated above.  Similarly,  if the plane $\Pi'$ generated by $P$ and $L$ contains a point of $\{P_1, P_2,P_3\}$, then by maximality, $S$ contains a $\M_{6b}$ configuration,  which case has been treated above. So we assume $\Pi'$ does not contain $P_1, P_2$ or $P_3$. In this case  $\{P_1, P_2, P_3, P\}$ is a $\M_{4b}$ configuration which we may represent as $\{e_0, e_1, e_2, e_3\}$. The line $L$ does not meet any of the $6$ lines $\overline{e_i,e_j}$, where $0\leq i<j \leq 3$. We may assume $L = \overline{e_0+e_1+e_2, e_0+c'e_1+e_3}$ where $c' \notin \{0,1\}$, i.e.  $c'=-1$. Here $I_2(T)$ has basis
\[  \{X_3(X_1+X_0 + X_2) ,  X_0(X_1-X_2+X_3),  X_1(X_0-X_2)+X_3(X_2+X_0)\},\]
 and $V(I_2(T))$ equals $T$. Thus, $T$ is a $\mC_7$. We denote, by $\M_{7e}$, the class of configurations which are $G$-equivalent to $T$. There are $|\M_{4b}|$ ways to choose the  $\M_{4b}$ configuration represented by $\{e_0, e_1, e_2, e_3\}$. Let $\Pi_i$ be the plane spanned by $\{e_0, e_1, e_2, e_3\} \setminus \{e_i\}$ and let $Q_i = \Pi_i \cap L$, $i=0,1,2,3$. Writing $Q_3=e_0+ae_1+be_2$ and $Q_2=e_0+ce_1+de_3$ the condition for the line $L = \overline{Q_1,Q_2}$ to not meet any of the $6$ lines $\overline{e_i,e_j}$ is $abcd \neq 0$ and $a \neq c$. Thus, there are $(3-1)^3(3-2)$ ways to pick $L$. Thus, $|\M_{7e}|=3^6(3^4-1)(3^3-1)(3^2-1)(3-2)/24$, i.e. $1516320$.

\subsubsection*{{\bf Classes $\M_{7f}$ and $\M_{7g}$}}\hfill \\
Next, we consider the case when $S$ contains the  $\M_{6e}$ configuration represented by $S_0=\{e_0, e_0+e_3, e_1, e_1+e_3,e_2, e_2+e_3\}$. Let $P_{2i+1}=e_i$ and  $P_{2i+2} = e_i + e_3$ for $i \in \{0,1,2\}$. Let $P_7$ be another point of $S$ and let $T = \{P, P_1, \dots, P_6\}$. There are $3$ possibilities for $P$: (i) $P_7$ lies on any of the lines $\overline{P_{2i+1},P_{2i+2}}$, $i \in \{0,1,2\}$ , (ii) $P_7$ is contained in any of the $3$ Planes $\Pi_0:\langle e_0,e_1,e_3\rangle$, $\Pi_1:\langle e_0,e_2,e_3\rangle$ and $\Pi_2:\langle e_1,e_2,e_3\rangle$ , (iii) $P_7$ lies in the complement of $\Pi_0 \cup \Pi_1 \cup \Pi_2$.

In case (i) and (ii), by maximality, $S$ contains  $\M_{6b}$,  which cases have already been treated.

In case (iii), we have  $P_7 = ae_0 +be_1+ce_2 +de_3$ with $abc \neq 0$. Here $I_2(T)$ has basis 
\[  \{X_1(aX_2-cX_0), X_0(bX_2-cX_1), bX_3(X_0+X_1+X_2-X_3)-d(a+b+c-d)X_0X_1\}. \]
A direct calculation shows that 
\[ \bar T = V(I_2(T)) =T \cup \{P'\}, \quad \text{where } P'= ae_0 +be_1+ce_2 +(a+b+c-d)e_3. \]
Thus, $T$ is a $\mC_7$ if and only if $2d = a+b+c$. Thus, there  are $4$ choices for $P_7$, namely \{ $(1,b,c, 1+b+c-d) : b,c \in \{\pm 1\}\}  $.  We denote the class of  configuration $G$-equivalent to $T$ by $\M_{7f}$.
In this case, we claim that the $\M_{6e}$ sub-configuration of $T=\M_{7f}$ is unique, in other words removing a point $Q$ from $T$ results in a $\M_{6e}$ configuration if and only if $Q=P_7$: We note that the $6$ points $P_1, \dots, P_6$  represented by  $e_i + \ep_i e_3$ where $i \in \{0,1,2\}$ and $\ep_i \in \{0,1\}$ 
generate $11$ planes of $PG(3,q)$. Of these $11$ planes there are $3$ planes which contain $4$ of $\{P_1, \dots, P_6\}$, these planes denoted $\Pi^i$ are  generated by $\{e_3,e_0,e_1,e_2\} \setminus \{e_i\}$ for $i=0,1,2$.
 The remaining $8$ planes are denoted $\Pi_{\ep}$ where $\ep= (\ep_0, \ep_1, \ep_2) \in \{0,1\}^3$, and are  generated by $\{e_0 + \ep_0 e_3, e_1 + \ep_1 e_3, e_2 + \ep_2 e_3 \}$. 
We recall from the discussion about the automorphism group of the $\M_{6e}$ configuration, that we may assume $P_7 = e_0+e_1+e_2+de_3$ for some $d \in \bF_3$. Since $2d = 1+1+1=0$, we must have $P_7 = e_0+e_1+e_2$. Clearly, of the $11$ planes above, $P_7$ is contained only in the planes $\Pi_{\ep}$ for $\ep=(0,0,0)$ and $\ep=(1,1,1)$ Thus,  $P_7$ lies on exactly two planes of the form $\Pi_\ep$ and $\Pi_{\ep}$. Therefore, none of the $6$ configurations  $\{P_1,\dots, P_7\} \setminus P_i$ for $i \in \{1, \dots, 6\}$ are $\M_{6e}$ configurations, Thus, proving that  $\{P_1, \dots, P_6\}$ is the unique $\M_{6e}$ configuration in $T$. Therefore, the number of $\M_{7f}$ configurations is $|\M_{6e}|\cdot 4$, i.e. $1010880$. We also  note that the automorphism group of $\M_{7f}$ configuration. We recall that the automorphism group of the $\M_{6e}$ configuration is $(Z_2 \times Z_2 \times Z_2)  \rtimes S_3$, where the $S_3$ factor permutes the lines $\overline{e_i,e_3}$ for $i \in \{0,1,2\}$ and the factor $Z_2 \times Z_2 \times Z_2$ has generators $g_0, g_2, g_2$ where $g_i$ interchanges the points $\{e_i, e_i+e_3\}$ for $i= 0,1,2$ respectively. Since $P_7=e_0+e_1+e_2$, the automorphism 
group of the $\M_{7f}$ configuration is $Z_2 \times S_3$, where the $Z_2$ factor is generated by the involution $g_0g_1g_2 \in Z_2 \times Z_2\times Z_2$.

We now consider the case \[ \bar T = V(I_2(T)) =T \cup \{P_8\}, \quad \text{where } P_8= ae_0 +be_1+ce_2 +(a+b+c-d)e_3. \]
The number of such pairs $\{P_7, P_8\}$ is $4$. We denote the class of  configuration $G$-equivalent to $\bar T$ by $\M_{7g}$. As above, we may assume $P, P'$ are given by $e_0+e_1+e_2 \pm  e_3$.
Thus, $P_7$ is on the planes $\Pi_{\ep}$ for $\ep=(1,0,0), (0,1,0), (0,0,1)$ and similarly $P_8$ is on the planes $\Pi_{\ep}$ for $\ep=(0,1,1), (1,0,1), (1,1,0)$. In other words, $P_7$ is on the planes generated by $\{P_2, P_3, P_5\}$, $\{P_1,P_4,P_5\}$, and $\{P_1,P_3,P_6\}$, and similarly $P_8$ is on the planes generated by $\{P_1, P_4, P_6\}$, $ \{P_2,P_3,P_6\}$, and $\{P_2,P_4,P_5\}$. Thus, a list of planes containing $4$ points of the configuration $T'$ is as below. We use the notation $(ijkl)$ to mean that $P_i, P_j,P_k,P_l$ are coplanar:
\[(1234), (1256), (3456), 
(1278), (5678), (3478)(2357), (1457), (1367), (1468),(2368),(2458). \]
Of the $\tbinom{8}{2}=28$ configurations $T_{ij} = \{P_1, \dots, P_8\} \setminus \{P_i, P_j\}$ obtained  from $T'$ by deleting $2$ of the points, we see that there are $16$ configurations of type $\M_{6e}$:
\[T_{12}, T_{14}, T_{16}, T_{17},  T_{23},  T_{25},
T_{28}, T_{34}, T_{36},
T_{37}, T_{45},T_{48}, T_{56}, T_{57}, T_{68}, T_{78},   \]
and there are $12$ configurations of type $\M_{6f}$:
\[T_{13}, T_{15},  T_{18},  T_{24}, T_{26}, T_{27},
T_{35}, T_{38}, T_{46},T_{47},
T_{58}, T_{67}.\]
While there are $4$ ways (choices of $ \{P_7,P_8\}$) to extend a given $\M_{6e}$ configuration to a $\M_{7g}$ configuration, we see that there are $16$ choices for the base $\M_{6e}$ configuration that yield the same $\M_{7e}$ configuration. Therefore, 
the number of $\M_{7g}$ configurations is given by $\frac{| \M_{6e}| \cdot 4}{16}$, i.e. $63180$.\\

Next, we consider the case when $S$ contains the  $\M_{6f}$ configuration represented by \[  S_0 = \{P_1=e_0,P_2= e_1, P_3= e_2, P_4=e_0+e_1+e_2, P_5=e_3, P_6=e_3+e_0-e_1\}. \]

We use the label   $(ijk)$  for the plane generated by $P_i, P_j, P_k$. If $P_l$ also lies on this plane, the label $(ijkl)$ will also denote the same plane.  The $6$ points of $S_0$ generate $14$ planes, of which only the planes $(1234)$ 
and $(1256)$ contain  $4$ points of $S_0$. These are the only planes containing both $P_1$ and $P_2$. Of the  remaining $12$ planes,  there are $8$ of them of the form $(ijk)$  for $i \in \{1,2\}, j \in \{3,4\}, k \in \{5,6\}$, and there are  $4$ planes $(345), (346), (356), (456)$.

We can write $PG(3,3)$ as the union of the $4$ planes containing the line $\overline{P_1,P_2}$ . These are the planes $(1234)$, $(1256)$, $\Pi_1, \Pi_{-1}$ , where $\Pi_1$ is the plane generated by $P_1, P_2, e_2+e_3$ and  $\Pi_{-1}$ is the plane generated by $P_1, P_2, e_2-e_3$.  Let $P_7$ be another point of $S$ and let $T= \{P_7, P_1, \dots, P_6\}$. There are $3$ possibilities for $P_7$: (i) some  $3$ points of $T$ are collinear, (ii) some $5$ points of $T$ are coplanar, (iii) no $3$ points of $T$ are collinear and no $5$ points of $T$ are coplanar.

In case (i) and (ii), by maximality, $S$ contains a $\M_{6b}$ configuration, which cases  which have already been treated. In case (iii) we represent $P_7 = ae_0 +be_1+ce_2 +de_3$.  The condition that no $5$ points of $T$ are coplanar,  implies that $P_7$ is not on the planes $(1234)$ and $(1256)$, i.e.  $cd \neq 0$. Therefore, we take $d=1$ and $c \in \{\pm  1\}$. We recall that the automorphism group of the  $\M_{6f}$ configuration is isomorphic to the group $Z_2 \times D_4$. The element $g_0$ generating the $Z_2$ factor  interchanges the planes $\Pi_1$ and $\Pi_{-1}$, so that we may assume $c=1$. We parametrize the points of the affine plane $\pi_1:=\Pi_1 \setminus \overline{P_1P_2}$ as $xe_0 +ye_1 + (e_2+e_3)$.

The $12$ planes $(ijk)$ (not containing both $P_1, P_2$) meet  $\pi_1$ in the twelve lines:
\[
\begin{split}
    &L_{345}: y=x, \quad   L_{346}: y=-x, \quad L_{356}: y=x+1, \quad L_{456}: y=-x-1.\\
 &L_{135}=L_{146}:y=0, \;L_{145}: y=1, \; L_{136}: y=-1. \\
& L_{235}:x=0, \;L_{245} = L_{236}: x=1, \; L_{246}: x=-1.
 \end{split}
\]
The $D_4$ factor generated by $g_1, g_2, g_3$ preserves $\Pi_1$ and the $13$ points of $\Pi_1$ form orbits under this group. One of these orbits consists of the $4$ points of $\overline{P_1, P_2}$. The remaining $9$ points are the points of the affine plane $\pi_1$. 
The four lines $L_{345}, L_{346}, L_{356}, L_{456}$ 
in $\pi_1$ form a  pair of parallel lines of slope $1$ and another pair of parallel lines of slope $-1$, which we refer to as a $\#$ pattern. These $4$ lines  account for $8$ of the $9$ points of $\pi_1$. The remaining point is $(x,y)=(0,1)$ which forms a single $D_4$-orbit. The four points of intersection of the pattern $\#$ are:
$(x,y) \in \{(0,0),(1,1), (1,-1),(-1,0)\}$ and they form the second orbit. Finally, the $4$ points of the $\#$ pattern which lie on exactly one of the four lines
form the third orbit: $(x,y) \in 
    \{ (0,1), (-1,-1), (0,-1), (-1,1) \}$. We summarize the three $D_4$ orbits on $\pi_1$:
\begin{enumerate}
     \item the four points with coordinates $(x,y) \in \{(0,0),(1,1), (1,-1),(-1,0)\}$ 
     which lie on one of the $4$ lines $\overline{P_iP_j}$ for $i \in \{3,4\}, j \in \{5,6\}$,
    \item the four points with coordinates $(x,y) \in 
    \{ (0,1), (-1,-1), (0,-1), (-1,1) \}$,
     which lie on  the intersection of three planes  \begin{enumerate}
        \item one of which is in the $4$ planes $(345), (346),(356), (456)$ 
        \item one of which is in the $4$ planes $(135), (145),(136), (146)$
        \item one of which is in the $4$ planes $(235), (245),(236), (246)$
    \end{enumerate}
    \item the point with coordinates $(x,y)=(1,0)$  which is  on  the intersection of four planes  \begin{enumerate}
        \item none of which is in the $4$ planes $(345), (346),(356), (456)$ 
        \item two of which are in the $4$ planes $(135), (145),(136), (146)$
        \item two  of which are in the $4$ planes $(235), (245),(236), (246)$.
    \end{enumerate}
\end{enumerate}
The condition that no $3$ points of $T$ are collinear implies that $P_7$ lies in $\pi_1$ and is  either in the second or the third orbit above. Thus, we may assume $P_7$ to be represented by $(x,y)=(0,1)$ or $(x,y)=(1,0)$.  In the former case, $P_7$ lies on the planes $(356),(145),(235)$. In this case, the configuration $T \setminus \{P_6\} = \{P_1,\dots, P_5, P_7\}$ is a $\M_{6e}$ configuration because it consists of the three non-coplanar lines $\overline{P_1P_4},  \overline{P_2P_3}, \overline{P_5P_7}$  which are concurrent through the point $e_1+e_2$. Since we have already accounted for the $\mC_7$ configurations arising from $\M_{6e}$, we skip this case. Therefore, we are left with $P_7$ represented by $e_0+e_2+e_3$.

 Here $I_2(T)$ has basis 
\[  \{-X_3( X_2 +X_1+ X_0 ), X_2(X_0-X_1 -X_3), X_0(X_2-X_3-X_1)\}. \]
Next we calculate $\bar T = V(I_2(T))$. If $Q = (x,y,z,w) \in \bar T$ and $zw=0$, then we just get $Q \in S_0$. So we take $Q=(x,y,z,1)$ with $z \in \{\pm 1\}$. A direct calculation shows that i) if $z=1$ then $Q = P_7$ and ii)  if $z = -1$ then the coordinates
$(x,y,-1,1)$ of $Q$ satisfy the conditions $ (x,y)=(0,-1)$, i.e.,  $Q=P_8 = -e_1 -e_2+e_3 = g_0 P_7$.
The planes containing $4$ points of the configuration
$\{P_1, \dots, P_8\}$ are 
\[ (1234), (1256),(1357),(1467), (2457), (2367),
(1458),(1368),  (2358), (2468), (3478),(5678). \]
Of the $\tbinom{8}{2}=28$  configurations $T_{ij} = \{P_1, \dots, P_8\} \setminus\{P_i, P_j\}$ we have $16$ $\M_{6e}$ configurations:
\[ T_{13},T_{14},T_{15},T_{16}, T_{23},T_{24},T_{25},T_{26},T_{37},T_{38}, T_{47}, T_{48}, T_{57}, T_{58}, T_{67}, T_{68},\]
and $12$ $\M_{6f}$ configurations:
\[ T_{12}, T_{17},T_{18},T_{27},T_{28}, T
_{34}, T_{35}, T_{36}, T_{45}, T_{46}, T_{56}, T_{78}. \]
Thus, $T'$ extends a $\M_{6e}$ configuration, which  has already been treated above (namely $\M_{7g}$ configuration).
\end{proof}

\subsection{$B_{j,7}$ for $j>7$}

\subsubsection*{}
An $\M_{7a}$ configuration $T$ represented by a plane $\Pi$ and a point $P$ outside $\Pi$, contains $14$ points of $PG(3,3)$. Since $j>7$, for $S \in \fC_i(T)$, $i\leq 6$, the only possible cases are $S=\M_{6a}, \M_{6b}$, and    $|\M_{6a}(T)|=1$,  $|\M_{6b}(T)|= \tbinom{13}{2}=78$.  By \eqref{eq:B_j_r_alpha}, we have, \[B_{j,7}(\M_{7a})=\tbinom{14}{j}-\tbinom{13}{j}-78\tbinom{8}{j}=\tbinom{13}{j-1}-78\tbinom{8}{j}.\]

\subsubsection*{}
An $\M_{7b}$ configuration $T$ represented by three concurrent lines not contained in a plane contains $10$ points of $PG(3,3)$. Since $j>7$, for $S \in \fC_i(T)$, $i\leq 6$, the only possible case is $S=\M_{6b}$,  and  $|\M_{6b}(T)|=\tbinom{3}{2}\cdot 3=9$ (as we have $\tbinom{3}{2}$ choices for two lines and $3$ choices for a point from the remaining line). By \eqref{eq:B_j_r_alpha}, we have, \[B_{j,7}(\M_{7b})= \tbinom{10}{j}-9\tbinom{8}{j}.\]

\subsubsection*{}
An $\M_{7c}$ configuration $T$ is represented by three lines $L_1,L_2,L_3$ of $PG(3,3)$, where $L_2$ and $L_3$ are skew to each other and the line $L_1$ meets both the lines.
It contains $10$ points of $PG(3,3)$. Since $j>7$, the only possible cases for $S \in \fC_i(T)$, $i\leq 6$ are $S=\M_{6b}, \M_{6c}$. As there are $2$ pairs of intersecting lines ($\{L_1,L_2\}$ and $\{L_2,L_3\}$) in $T$, and a point can be chosen from the remaining line in $3$ ways, we have $|\M_{6b}(T)|=6$. It is clear that  $|\M_{6c}(T)|=1$, as there is only one configuration of a  pair of skew lines and which is $\{L_1,L_3\}$. By \eqref{eq:B_j_r_alpha}, we have, \[B_{j,7}(\M_{7c})=\tbinom{10}{j}-6\cdot \tbinom{8}{j}-\tbinom{8}{j}= \tbinom{10}{j}-7\cdot \tbinom{8}{j}.\]

\subsubsection*{}
A $\M_{7d}$ configuration $T$ represented by $(\{L_1,L_2\},\{P_1,P_2\})$, where $\{L_1,L_2\}$ is a pair of intersecting lines and the line $\overline{P_1,P_2}$ not contained in the plane does not meet $L_1$ and $L_2$, contains $9$ points of $PG(3,3)$. Since $j>7$, the only possible case for $S \in \fC_i(T)$, $i\leq 6$ is $S=\M_{6b}$. Clearly, $| \M_{6b}(T)|=2$, as there is one pair of intersecting lines and $2$ choices for a point from $\{P_1,P_2\}$. By \eqref{eq:B_j_r_alpha}, we have, \[B_{j,7}(\M_{7d})=\tbinom{9}{j}-2\cdot \tbinom{8}{j}.\]

\subsubsection*{}
Let $T$ be a $\M_{7e}$ configuration represented by $(\{P_1,P_2,P_3,P_4\},L)$, where the points $P_1,P_2,P_3,P_4$ are in general position in $PG(3,3)$ and the line $L$ does not meet any of the $6$ lines $\overline{P_i,P_j}$, where $1\leq i < j \leq 4$. The configuration $T$ has $8$ points of $PG(3,3)$. In this case,  $\fC_i(T) =\varnothing$ for  $i\leq 6$. By \eqref{eq:B_j_r_alpha}, we have \[B_{8,7}(\M_{7e})=1.\]

\subsubsection*{}
Let $T$ be a $\M_{7g}$ configuration represented by $\{P_1, \dots, P_6,P_7,P_8\}$ where the $4$ lines $L_i = \overline{P_{2i-1}, P_{2i}}$, $i=1,2,3,4$, meet at a single point $Q$ different from $\{P_1, \dots, P_6,P_7,P_8\}$ and no three of them are contained in a plane. The configuration $T$ has  $8$ points of $PG(3,3)$. Since $j>7$, we have no such $S$ with $S \in \fC_i(T)$, $i\leq 6$. By \eqref{eq:B_j_r_alpha}, we have, \[B_{8,7}(\M_{7g})=1.\]

Since a $\M_{7f}$ configuration contains exactly $7$ points $PG(3,3)$ and $j>7$, we get $B_{j,7}(\M_{7f})=0$. Therefore, by \eqref{eq:B_j_r}, for $j>7$, we have

\begin{equation*}
\begin{split}
      B_{8,7}&=1080 \cdot (\tbinom{13}{7}-78\tbinom{8}{8})+9360\cdot (\tbinom{10}{8}-9\tbinom{8}{8})\\
    &+84240 \cdot (\tbinom{10}{8}-7\cdot \tbinom{8}{8})+ 505440\cdot \tbinom{9}{8}-2\cdot \tbinom{8}{8}+1516320+63180,
\end{split}
\end{equation*}
and, if $j\geq 9$,

\begin{equation*}
      B_{j,7}=1080 \cdot (\tbinom{13}{j-1})+9360\cdot \tbinom{10}{j}
    +84240 \cdot \tbinom{10}{j}+ 505440\cdot \tbinom{9}{j},
\end{equation*}
 i.e.
\begin{equation}\label{BJ7}
B_{8,7}=9413820, \text{ and } j\geq 9, B_{j,7}=1080 \cdot \tbinom{13}{j-1}+93600 \cdot \tbinom{10}{j}+ 505440\cdot \tbinom{9}{j}.
\end{equation}

\section{  Configurations of rank $8$ }
\label{r8}
As before, we build the $\mC_8$'s by adding points to $\mC_7$'s. We carry this out systematically by running through the classes in $\fC_7$.
If $T$ is a $\mC_8$, then $I_2(T)$ is a $2$ dimensional vectors space, and its projectivization is a pencil of quadratic forms on $PG(3,3)$. If $S$ denote the set of $(q+1)=4$ elements of this pencil, let $\nu_i(T)=|S \cap \mO_i|$, here $\mO_1, \dots, \mO_6$ are the $G$-orbits of quadratic forms on $PG(3,3)$, listed in Table \ref{tab:table1}.  We will write down the numbers $\nu(T)=(\nu_1(T), \dots, \nu_6(T))$.
\begin{lemma} \label{max8}
The set $\fC_8$ can be partitioned into the following $4$ classes:  
\begin{enumerate}
\item [$\M_{8a}$] of size $2^3 \cdot 3^2 \cdot 5 \cdot 13$. This class consists of a plane of $PG(3,3)$ and a line not contained in that plane.
\item [$\M_{8b}$] of size 
$2^4 \cdot 3^2 \cdot 5 \cdot 13$.  This class consists of $4$ concurrent lines such that no three are contained in a plane.

\item [$\M_{8c}$] of size 
 $2^2 \cdot 3^6 \cdot 5 \cdot 13$. This class consists of $4$ lines $L_1,L_2,L_3, L_4$ such that $L_i$ and $L_j$ intersect if and only if $i \neq j \mod 2$ (and $L_i$ and $L_j$ are skew if and only if $i=j \mod 2$).\\

\item [$\M_{8d}$] of size $2^6 \cdot 3^4 \cdot 5 \cdot 13$. This class consists of $(\{L_1,L_2\},\{P_1,P_2,P_3\})$ where $\{ L_1 ,L_2\}$ is a pair of lines in a plane $\pi_1$, and the points $P_1, P_2, P_3$ are in a different plane $\pi_2$. The point $Q = L_1\cap L_2$ lies on the line $L=\pi_1 \cap \pi_2$, but none of the points $P_1, P_2, P_3$ lie on $L$. The four points $\{Q, P_1, P_2, P_3\}$ are in general position in $\pi_2$. \\

\item [$\M_{8e}$] of size  $2^3 \cdot 3^6 \cdot 5 \cdot 13$. This class consists of configurations $\{P_1, \dots, P_6,P_7, P_8\}$ where the $3$ lines $L_i = \overline{P_{2i-1} ,P_{2i}}$, $i=1,2,3$, are non-coplanar lines meeting at a single point $Q$ different from $\{P_1, \dots, P_6\}$.  There are four lines $L$ passing through $Q$ with the property that no three of the four lines $\{L_1,L_2,L_3,L\}$ are coplanar.  Let $L_4, L_5$ be two such lines. On each of these lines $\{L_4, L_5\}$ there is a unique point $\{Q_1(L_4), Q_1(L_5)\}$ as given in the description of the  $\M_{7f}$ configuration (see Lemma  \ref{list7}). We take $\{P_7, P_8\}=\{Q_1(L_4), Q_1(L_5)\}$.

\end{enumerate}
\end{lemma}

\begin{proof}
Let $S$ be a $\mC_8$  configuration. 

\subsubsection*{{\bf Class $\M_{8a}$}}\hfill \\
Let $S$ contain a $\M_{7a}$ configuration represented by the point set $T = \{e_3\} \cup \Pi$, where the plane $\Pi$ given by $X_3=0$. If $S$ contains a point $P\notin T$, then $P$ will be on a line $\overline{e_3,Q}$, where $Q$ is a point of $\Pi$. Then by maximality, $\overline{e_3,P}$ is contained in $S$. We can assume $P=e_0+e_1+e_3$ and let $T^\prime = \overline{e_3,e_0+e_1+e_3} \cup \Pi$. Here $I_2(T^\prime)$ has basis $\{X_3(X_0-X_1), X_3X_2\}$ and $V(I_2(T^\prime))$ clearly equals $T^\prime$. Thus, $T^\prime$ is a $\mC_8$ which we denote $\M_{8a}$. Clearly, all the $(q+1) =4$ elements of the pencil $\bP(I_2(T'))$ are in $\mO_1$ (pairs of planes). Thus, $\nu(T')=(4,0,0,0,0,0)$.
The number of ways to choose the plane $\Pi$ is $40$ and a line not contained in $\Pi$ is $13\cdot 3^2$. Therefore, the size of $\M_{8a}$  is $40\cdot 13\cdot 3^2= 4680$.

\subsubsection*{{\bf Class $\M_{8b}$}}\hfill \\
Let $S$ contain a $\M_{7b}$ configuration represented by $S_0=\overline{e_0,e_1}\cup \overline{e_0,e_2} \cup \overline{e_0,e_3}$. Let $P$ be  another point of $S$. There are two possibilities for $P$: (i) $P$ lies on one of the planes $\Pi_1:X_1=0$, $\Pi_2:X_2=0$, $\Pi_3:X_3=0$, and (ii) $P$ lies outside $\Pi_1,\Pi_2$ and $\Pi_3$.

In case (i) by maximality, $S$ contains a $\M_{7a}$ configuration which case has been treated above.

In case (ii), if $P=ae_0+be_1+ce_2+de_3$, then the condition that $P$ does not lie on the planes $\Pi_1,\Pi_2$ and $\Pi_3$ is that $bcd\neq 0$. So we can assume $d=1$ and $P=ae_0+be_1+ce_2+e_3$. Let $T^\prime = \overline{e_0,e_1}\cup \overline{e_0,e_2} \cup \overline{e_0,e_3} \cup \{P\}$ . Here $I_2(T^\prime)$ has basis $\{X_3(cX_1-bX_2), X_2(X_1-bX_3)\}$ and $V(I_2(T^\prime))$ equals $T=\overline{e_0,e_1}\cup \overline{e_0,e_2} \cup \overline{e_0,e_3} \cup \overline{e_0,P} $. Thus, $T$ is a $\mC_8$ which we denote $\M_{8b}$. 
In the pencil $\bP(I_2(T))$, apart from the mentioned generators, the remaining $(q-1)=2$ elements are in $\mO_4$. Thus, $\nu(T)=(2,0,0,2,0,0)$. 

The number of ways a $\M_{7b}$ configuration $S_0$ can be chosen is $|\M_{7b}|$. Since no three lines of a $\M_{8b}$ configuration are contained in a plane, there are $4$ choices for the line $\overline{e_0,P}$. The configuration $T\setminus \overline{e_0,e_i}$, $i=1,2,3$ is $\M_{7b}$. Thus, there are $4$ different $\M_{7b}$ configurations contained in a $\M_{8b}$ configuration. Therefore, the size of $\M_{8b}$ is $\frac{|\M_{7b}|\cdot 4}{4}=|\M_{7b}|=9360$.

\subsubsection*{{\bf Class $\M_{8c}$}}\hfill \\
Let $S$ contain a $\M_{7c}$ configuration represented by $S_0=\overline{e_0,e_1}\cup \overline{e_0,e_2}\cup \overline{e_1,e_3}$.  Let $P$ be another point of $S$. There are two possibilities for $P$: (i) $P$ lies on one of the planes $\Pi_1:X_3=0$, $\Pi_2:X_2=0$, and (ii) $P$ lies outside $\Pi_1,\Pi_2$. In case (i)  by maximality,  $S$ contains a $\M_{7a}$ configuration which case has been treated above. In case (ii), if $P=ae_0+be_1+ce_2+de_3$, then the condition that $P$ does not lie on the planes $\Pi_1,\Pi_2$ is that $cd\neq 0$. Switching to the basis $f_0=e_0, f_1=e_1, f_2=ce_2+ae_0, f_3=de_3+be_1$ of $PG(3,3)$, the configuration $S_0$ is again given by $S_0=\overline{f_0,f_1}\cup \overline{f_0,f_2}\cup \overline{f_1,f_3}$, and $P = f_2+f_3$. Thus, we may take $T'=S_0 \cup \{P\}$ to be  $T^\prime = \overline{e_0,e_1}\cup \overline{e_0,e_2} \cup \overline{e_1,e_3} \cup \{e_2+e_3\}$ . Here $I_2(T^\prime)$ has basis $\{ X_3 X_0, X_2X_1\}$ and $V(I_2(T^\prime))$ equals $T=\overline{e_0,e_1}\cup \overline{e_0,e_2} \cup \overline{e_1,e_3} \cup \overline{e_2,e_3}$. Thus, $T$ is a $\mC_8$ which we denote $\M_{8c}$. In the pencil $\bP(I_2(T))$, apart from the mentioned generators, the remaining $(q-1)=2$ elements are in $\mO_2$. Thus, $\nu(T)=(2,2,0,0,0,0)$.
The size of $\M_{8c}$ is $3 \cdot |\M_{4b}| = 189540$, because there are $3$ ways to obtain the configuration $T$ from four points in general position of $PG(3,3)$.


\subsubsection*{{\bf Class $\M_{8d}$}}\hfill \\
Let $S$ contain a $\M_{7d}$ configuration represented by $S_0=\overline{e_0,e_1}\cup \overline{e_0,e_2}\cup \{e_3, e_3+e_1+e_2\}$. If $P$ is another point of $S$, the there are $4$ possibilities for $P$: (i) $P$ lies in the plane $\Pi_0:\langle e_0,e_1,e_2\rangle$, (ii) $P$ lies in at least one of the planes $\Pi_1:\langle e_3,e_0,e_1\rangle$,  $\Pi_2:\langle e_3,e_0,e_2\rangle$, $\Pi_3:\langle e_1+e_2+e_3,e_0,e_1\rangle$, $\Pi_4:\langle e_1+e_2+e_3,e_0,e_2\rangle$, (iii) $P$ lies on the line $L:\overline{e_3,e_1+e_2+e_3}$, (iv) $P$ lies in the complement of $\Pi_0\cup \Pi_1\cup \Pi_2\cup \Pi_3\cup \Pi_4\cup L$ in $PG(3,3)$.

In the cases (i) and (iii) by maximality, $S$ contains a $\M_{7a}$ configuration, which case has been treated above.

In the case (ii) by maximality, $S$ contains a $\M_{7b}$ or $\M_{7c}$ configuration, which case has been treated above.

In case (iv), let $P=ae_0+be_1+ce_2+de_3$. The condition for $P\notin \Pi_0$ is that $d\neq 0$. So we can assume $d=1$. The conditions for $P\notin \Pi_1, \Pi_2,\Pi_3,\Pi_4$ are $c\neq 0,1$ and $b\neq 0,1$.  Therefore $b=c=-1$. The condition for $P$ to lie on the line $L$ is $a=b-c=0$ . Thus, we may take  $P$ to be  $ae_0-e_1-e_2+e_3$  with  $a\in \{ \pm 1\}$. Let $T^\prime =\overline{e_0,e_1}\cup \overline{e_0,e_2}\cup \{e_3, e_3+e_1+e_2, P\}$. Here $I_2(T^\prime)$ has basis  
\[ \{aX_1(X_3-X_2)-X_0X_3, X_3X_1+X_2X_1+X_3X_2\},\]
and $V(I_2(T^\prime))$ equals $T^\prime$. Thus,  $T^\prime$ is a $\mC_8$ which we denote $\M_{8d}$.  Of the $(q+1)=4$ points of the pencil $\bP(I_2(T'))$,  except for the element   $X_3X_1+X_2X_1+X_3X_2$ 
 which is in $\mO_4$,  the remaining $q=3$ elements are in $\mO_2$ . Thus, $\nu(T')=(0,3,0,1,0,0)$.

We will now count the size of $\M_{8d}$. This configuration consists of a pair of intersecting lines $L_1,L_2$ in a plane $\Pi$, and three non-collinear points $P_1,P_2,P_3$ lying outside $\Pi$,  such that the $3$ lines formed by $P_1,P_2,P_3$ do not meet $L_1,L_2$. Since  $q=3$,  with this condition, the plane $\langle P_1,P_2,P_3\rangle$ meets $\Pi$ in a line $M$ through $L_1\cap L_2$. A pair of intersecting lines $L_1,L_2$ in $PG(3,3)$ can be chosen in $\M_{5a}$ ways, the line $M$ can be chosen in $\tbinom{2}{1}$ ways in the plane $\langle L_1,L_2\rangle$, a plane $\Pi'$ through $M$ other than $\Pi$ can be chosen in $\tbinom{3}{2}$ ways and on $\Pi'$ a set of $3$ noncollinear points with the property no two are lying on a line through $L_1\cap L_2$, can be chosen in $(3^3-3^2)$ ways. Hence, the size of $\M_{8d}$ is $|\M_{5a}|\cdot \tbinom{2}{1}\cdot \tbinom{3}{2} \cdot (3^3-3^2)$, i.e. $336960$.\\

Let $S$ contain a $\M_{7e}$ configuration represented by $S_0=\{e_0, e_1, e_2, e_3\}\cup \overline{e_0+e_1+e_2, e_0-e_1+e_3}$. Let us denote $P_1=e_0+e_1+e_2$ and $P_2=e_0-e_1+e_3$. All the points of $PG(3,q)$ are on the $4$ planes $\Pi_0:\langle P_1,P_2,e_0\rangle$, $\Pi_1:\langle P_1,P_2,e_1\rangle$, $\Pi_2:\langle P_1,P_2,e_2\rangle$, $\Pi_3:\langle P_1,P_2,e_3\rangle$, through $\overline{P_1,P_2}$. Let $P$ be another point of $S$. Then $P$ will lie on one of these $4$ planes, say $\Pi_0$. In this case, the plane $\Pi_0$ contains $5$ points of $S$, and by maximality, $S$ contains a $\M_{7d}$ configuration, which case has been treated above.

\subsubsection*{{\bf Class $\M_{8e}$}}\hfill \\
Let $S$ contain a $\M_{7f}$ configuration represented by $S_0=\{e_0,e_0+e_3,e_1,e_1+e_3,e_2,e_2+e_3, e_0 +e_1+e_2 \}$.  Here,  $I_2(S_0)$ has a basis \[  \{X_1(X_2-X_0), X_0(X_2-X_1), X_3(X_0+X_1+X_2-X_3)\}. \]

Let $P'=ae_0+be_1+ce_2+de_3$ be another point of $S$. Then there are three possibilities for $P'$: (i) $P'$ lies on one of the planes $\Pi_0:\langle e_0,e_1,e_3\rangle$, $\Pi_1:\langle e_0,e_2,e_3\rangle$ and $\Pi_2:\langle e_1,e_2,e_3\rangle$, $\Pi_3:\langle e_0,e_1,e_2\rangle$, $\Pi_4:\langle e_0+e_3,e_1+e_3,e_2+e_3\rangle$ containing $4$ points of $S_0$, (ii) $P'$ lies on one of the 21 lines containing two points of $S_0$, (iii) $P'$ lies on the line $\overline{e_3,e_0+e_1+e_2}$, (iv) $P'$ lies in the complement of these planes and lines in $PG(3,3)$. 

In cases (i) and (ii), by maximality, $S$ contains a $\M_{6b}$ configuration. From a $\M_{6b}$ configuration, we have seen earlier that a $\M_{7b}$, $\M_{7c}$ or $\M_{7d}$ configuration can be obtained, and these cases have been treated above. 

In case (iii), if $P'$ lies on the line $\overline{e_3,e_0+e_1+e_2}$, then $P'$ is either $e_0+e_1+e_2+e_3$ or $-e_0-e_1-e_2-e_3$. Thus,  the $8$ points of $S_0\cup \{P'\}$ lies on $4$ concurrent lines $\overline{e_3e_0}$, $\overline{e_3e_1}$, $\overline{e_3e_2}$, $\overline{e_3e_0+e_1+e_2}$. But these $8$ points of $S_0\cup \{P'\}$ do not satisfy the condition to be a $\M_{7g}$ configuration. So, the configuration consisting of points of $S_0\cup \{P'\}$ is of rank $8$ and contained in a $\M_{8b}$ configuration represented by the $4$ concurrent lines $\overline{e_3e_0}$, $\overline{e_3e_1}$, $\overline{e_3e_2}$, $\overline{e_3e_0+e_1+e_2}$. By Lemma \ref{maxl_uniq}, $S_0\cup \{P'\}$ is not a maximal configuration and $S$ is a $\M_{8b}$ configuration.

Now we will consider case (iv). The number of points contained in the union $\Pi_0\cup \Pi_1\cup \Pi_2\cup \Pi_3\cup \Pi_4$ is $35$. All the $21$ lines formed by the $2$
 points of $S_0$ are contained in these $5$ planes. Also, there are two points on the line $\overline{e_3,e_0+e_1+e_2}$ which are not in the union of these $5$ planes and $21$ lines. Thus, there are total $37$ points of $PG(3,3)$ covered by union of the planes $\Pi_0, \Pi_1, \Pi_2,\Pi_3,\Pi_4$, and $\overline{e_3,e_0+e_1+e_2}$.  The remaining   $3$  points 
 $\{e_0-e_1-e_2+e_3,-e_0+e_1-e_2+e_3,-e_0-e_1+e_2+e_3\}$ which form a single orbit under the automorphism group $Z_2 \times S_3$ of the $\M_{7f}$ configuration. So we may take $P'= e_0-e_1-e_2+e_3$.  Let $T=S_0\cup \{P'\}$. A basis of the  $I_2(T)$ is given by 
\[ \{X_1(X_2-X_0)+X_3(X_0+X_1+X_2-X_3), X_0(X_2-X_1)\},\] and it is readily checked that $V(I_2(T))$ equals $T$. Thus, $T$ is a maximal $\mC_8$, which we denote as $\M_{8e}$. The configuration $T_i = T \setminus \{P_i\}$ for $i=1, \dots ,8$ are $\mC_7$ configurations with $7$ points and containing no lines. In the list $\M_{7a}-\M_{7g}$, the only configuration with $7$ points and containing no lines is $\M_{7f}$.
Therefore, the size of the class of $\M_{8e}$ configurations  is:
\[ |\M_{8e}| = \tfrac{1}{8} |\M_{7f}| \cdot 3 = |G|/32=379080.\]

In the pencil $\bP(I_2(T))$, there is one element each in $\mO_1$ and $\mO_2$  and the remaining $2$ elements are in $\mO_6$.  Thus, $\nu(T)=(1,1,0,0,0,2)$.\\

Let $S$ contain a $\M_{7g}$ configuration represented by $S_0=\{e_0,e_0+e_3,e_1,e_1+e_3,e_2,e_2+e_3, e_0+e_1-e_2, e_0+e_1-e_2+e_3\}$. Let us consider the six planes $\Pi_0:\langle e_0,e_1,e_3\rangle$, $\Pi_1:\langle e_2,e_0+e_1-e_2,e_3\rangle$, $\Pi_2:\langle e_0,e_2,e_3\rangle$, $\Pi_3:\langle e_0,e_0+e_3,e_0+e_1-e_2\rangle$, $\Pi_4:\langle e_1,e_2,e_3\rangle$, and $\Pi_5:\langle e_1,e_1+e_3,e_0+e_1-e_2\rangle$. Each of the plane $\Pi_i,i\in\{0,\dots,5\}$, contain $4$ points of $S_0$. Now,

\begin{align*}
    |\cup_{i=0}^{5} \Pi_i|&=|\Pi_0|+|\Pi_1\setminus \Pi_0|+|\Pi_2\setminus(\Pi_0\cup\Pi_1)|+|\Pi_3\setminus(\cup_{i=0}^{2}\Pi_i)|+|\Pi_4\setminus(\cup_{i=0}^{3}\Pi_i)|+|\Pi_5\setminus(\cup_{i=0}^{4}\Pi_i)|\\
    &=13+9+6+6+3+3=40.
\end{align*}

This shows that the union $\cup_{i=0}^{5} \Pi_i$ contains all the points of $PG(3,3)$. If $S$ contains another point $P$, then $P$ will be on at least one of the planes $\Pi_i$ and since $\Pi_i$ contains $5$ points of $S$, by maximality, it will contain a $\M_{5b}$ configuration. From a $\M_{5b}$ configuration we only get a$\M_{6b}$ configuration and from a $\M_{6b}$ configuration we can get a $\M_{7b}$, a $\M_{7c}$ or a $\M_{7d}$ configuration. So, this case has been treated above.
\end{proof}

\subsection{$B_{j,8}$ for $j>8$}
\subsubsection*{}

Let $T$ be a $\M_{8a}$ configuration of $16$ points  represented by a plane $\Pi$ of $PG(3,3)$ and a line $L$ not contained in $\Pi$. Let $L\cap \Pi=\{Q\}$. Since $j>8$, the possible cases for $S \in \fC_i(T)$, $i\leq 7$ are  as follows:
\begin{enumerate}
    \item[(i)] $S=\M_{7a}$ consisting of a plane $\pi$ and a point $P$ not on $\pi$. In this case $|\M_{7a}(T)|=3$, as we must take $\pi = \Pi$ and there  are $3$ ways to pick $P$ from $L\setminus Q$.
    \item[(ii)]  $S=\M_{7b}$ consisting of $3$ non-coplanar lines $\{\ell_1, \ell_2,\ell_3\}$ through a point $P$. In this case, $|\M_{7b}(T)|=6$, as we must take $P=Q$, $\ell_1 = L$, and there are $\tbinom{4}{2}=6$ choices for $\{\ell_2, \ell_3\}$ in $\Pi$.
    \item[(iii)]  $S=\M_{7c}$ consisting of three lines $\{\ell_1, \ell_2,\ell_3\}$ where $\ell_2$  and $\ell_3$ are skew to each other and the line $\ell_1$  meets both the lines.   We must take $\ell_3=L$, $\ell_2$ to be a line in $\Pi$ not passing through $Q$, and $\ell_1$ a line in $\Pi$ passing through $P$. Thus, there are $1$, $q^2=9$ and $(q+1)=4$ choices for $\ell_1, \ell_2$ and $\ell_3$ respectively. Therefore, $|\M_{7c}(T)|=36$.
    \item[(iv)]  $S=\M_{7d}$ which consists of  a pair of  lines $\ell, \ell'$ in a plane $\pi$ , and a pair of points $P,P'$  not lying on $\pi$ such that the line $\overline{P, P'}$  meets $\pi$ in a point not lying on $\ell \cup \ell'$.  In this case $|\M_{7d}(T)|=108$, as there are $\tbinom{9}{2}=36$ choices for the  pair $\ell, \ell'$ of lines not passing through $Q$ and $\tbinom{3}{2}=3$ choices for the $2$ points $P, P'$  from $L\setminus Q$.
    \item[(v)]  $S=\M_{6a}$ consisting of a plane $\pi$. In this case  we must take $\pi = \Pi$ and hence $|\M_{6a}(T)|=1$.
\end{enumerate}

By \eqref{eq:B_j_r_alpha}, we have, \[B_{j,8}(\M_{8a})=\tbinom{16}{j}-3\cdot (\tbinom{14}{j}-\tbinom{13}{j})-6\cdot \tbinom{10}{j}-36\cdot \tbinom{10}{j}-108\cdot \tbinom{9}{j}-\tbinom{13}{j},\]
i.e. 
\[B_{j,8}(\M_{8a})=\tbinom{16}{j}-3\cdot \tbinom{13}{j-1}-42\cdot \tbinom{10}{j}-108\cdot \tbinom{9}{j}-\tbinom{13}{j}.\]

\subsubsection*{}
A $\M_{8b}$ configuration $T$ of $13$ points represented by $4$ concurrent lines $L_1, \dots, L_4$ through a point $Q$, such that no three of these lines are coplanar. Since $j>8$, the possible cases for $S \in \fC_i(T)$, $i\leq 7$ are  listed as follows.
\begin{enumerate}
    \item[(i)]  $S=\M_{7b}$ consisting of three non-planar concurrent lines. In this case $|\M_{7b}(T)|=4$, as there $\tbinom{4}{3}$ ways to choose $3$ concurrent lines in $T$.
    \item[(ii)]  $S=\M_{7d}$ which consists of  a pair of  lines $L, L'$ in a plane $\pi$, and a pair of points $P,P'$  not lying on $\pi$ such that the line $\overline{P, P'}$  meets $\pi$ in a point not lying on $L \cup L'$. In this case $|\M_{7d}(T)|=\tbinom{4}{2}\cdot 3\cdot 3=54$, as there are $\tbinom{4}{2}$ ways to choose to pick $L, L'$ from $\{L_1, \dots, L_4\}$, and $9$ ways to pick a point each (different from $Q$) from the remaining two lines of $\{L_1, \dots , L_4\} \setminus \{L, L'\}$.
    \end{enumerate}
By \eqref{eq:B_j_r_alpha}, we have, \[B_{j,8}(\M_{8b})=\tbinom{13}{j}-4\cdot \tbinom{10}{j}-54\cdot \tbinom{9}{j}.\]

\subsubsection*{}
Let $T$ be a $\M_{8c}$ configuration of $13$ points  represented by $4$ lines $L_1,L_2,L_3, L_4$ such that $L_i$ meets $L_j$ if and only if $i \neq j \mod 2$ .
Since $j>8$, the possible cases for $S \in \fC_i(T)$, $i\leq 7$ are   as follows.
\begin{enumerate}
    \item[(i)]  $S=\M_{7c}$  consisting of three lines $\{\ell_1, \ell_2,\ell_3\}$ where $\ell_2$  and $\ell_3$ are skew to each other and the line $\ell_1$  meets both the lines.  In this case $|\M_{7c}(T)|=4$, as we must take $S$ to be the union of any $3$ of the $4$ lines $\{L_1, \dots, L_4\}$ . 
    \item[(ii)]  $S=\M_{7d}$ which consists of  a pair of  lines $L, L'$ in a plane $\pi$, and a pair of points $P,P'$  not lying on $\pi$ such that the line $\overline{P, P'}$  meets $\pi$ in a point not lying on $L \cup L'$. In this case, we must take $L, L'$ to a pair $L_i, L_j$ with $i \neq j \mod 4$ (of which there are $4$ choices). If $\{L,L'\} = \{L_1, L_2\}$ then there are $2$ choices of points  on  each of the lines $L_3, L_4$ for the pair $P, P'$.
    Thus, $|\M_{7d}(T)|=4\cdot 2\cdot 2=16$
\end{enumerate}
By \eqref{eq:B_j_r_alpha}, we have, \[B_{j,8}(\M_{8c})=\tbinom{12}{j}-4\cdot \tbinom{10}{j}-16\cdot \tbinom{9}{j}.\]

\subsubsection*{}
Let $T$ be a $\M_{8d}$ configuration of $10$ points, represented by $(L_1,L_2,\{P_1,P_2,P_3\})$ where $\{ L_1 ,L_2\}$ is a pair of lines in a plane $\pi_1$, and the points $P_1, P_2, P_3$ are in a different plane $\pi_2$. The point $Q = L_1\cap L_2$ lies on the line $L=\pi_1 \cap \pi_2$, but none of the points $P_1, P_2, P_3$ lie on $L$. The four points $\{Q, P_1, P_2, P_3\}$ are in general position in $\pi_2$.  Since $j>8$, the possible case for $S \in \fC_i(T)$, $i\leq 7$ is $S=\M_{7d}$. In this case $|\M_{7d}(T)|=3$, as we must take $S = T \setminus \{P_i\}$ for $i = 1\dots 3$. By \eqref{eq:B_j_r_alpha}, we have, \[B_{j,8}(\M_{8d})= \tbinom{10}{j}-3\cdot \tbinom{9}{j}.\]

We have, $B_{j,8}(\M_{8e})=0$, as an $\M_{8e}$ configuration contains exactly $8$ points $PG(3,3)$. Therefore, by \eqref{eq:B_j_r}, for $j>8$, 

\begin{equation*}
\begin{split}
B_{j,8}&= 4680 \cdot (\tbinom{16}{j}-3\cdot \tbinom{13}{j-1}-42\cdot \tbinom{10}{j}-108\cdot \tbinom{9}{j}-\tbinom{13}{j})+9360\cdot (\tbinom{13}{j}-4\cdot \tbinom{10}{j}-54\cdot \tbinom{9}{j})\\
&+189540\cdot (\tbinom{12}{j}-4\cdot \tbinom{10}{j}-16\cdot \tbinom{9}{j})+336960\cdot (\tbinom{10}{j}-3\cdot \tbinom{9}{j}),
\end{split}
\end{equation*}
i.e.,
\begin{equation*}\label{BJ8}
B_{j,8}= 4680 \cdot \tbinom{16}{j}+4680\cdot \tbinom{13}{j}-14040\cdot \tbinom{13}{j-1}-655200\cdot \tbinom{10}{j}-5054400\cdot \tbinom{9}{j}.
\end{equation*}

\section{Configurations of rank $9$}
\label{r9}
A configuration  $S$ of rank $9$ is,  by definition, contained  in the set of points $\mQ$ of a quadric of $PG(3,3)$ defined by some quadratic form $f(X_0, \dots, X_3)$.
The points of $\mQ$ are a rank $9$ configuration if and only if $I_2(\mQ)$ is one dimensional. As seen in Table \ref{tab:table1}, there are six $G$-orbits of quadratic forms on $PG(3,3)$  represented by 
\[\varphi_1=X_0X_1,\, \varphi_2=X_0X_1-X_2X_3, \, \varphi_3= X_1^2 + X_3^2- X_0X_2, \,
\varphi_4=X_0X_2-X_1^2, \, \varphi_5=X_0^2, \,\varphi_6=X_0^2 + X_1^2.\]
The quadric $\mQ_i = V(\varphi_i)$ is a rank $9$ configuration if and only if $I_2(V(\varphi_i))$ is the one dimensional vector space generated by $\varphi_i$.
This is not the case for i) $\varphi_6$, ii) $\varphi_5$ and iii) $\varphi_4$ because: i) $V(\varphi_6)$ is a line which has rank $3$, ii) $V(\varphi_5)$ is a plane which has rank $6$, and iii) $I_2(V(\varphi_4))$ is $2$-dimensional vector space spanned by $\varphi_4$ and $X_0(X_2-X_1)$. 

 The quadrics $\mQ_1, \mQ_2$ and $\mQ_3$ have $22, 16$ and $10$ points respectively. The configurations  $\M_{8a}, \M_{8c}$ and $\M_{8e}$ have $16, 12$ and $8$ points respectively.
We recall from \S \ref{r8}, that the pencil of quadrics defining $\M_{8a}, \M_{8c}$ and $\M_{8e}$ configurations contains an element of  $G \cdot \varphi_1$ , $G \cdot \varphi_2$  and $G \cdot \varphi_3$ respectively. Therefore, i) $\mQ_1$  properly contains an $\M_{8a}$ configuration, ii) $\mQ_2$ properly contains an $\M_{8c}$ configuration, and iii) $\mQ_3$ properly contains an $\M_{8e}$ configuration. Since the containment is proper in each case, and the configurations $\M_{8a}, \M_{8c}, \M_{8e}$ are maximal, it follows that the quadrics $\mQ_1, \mQ_2, \mQ_3$ are $\mC_9$'s.
We summarize this:
\begin{lemma}
The set $\fC_9$ can be partitioned into the following 3 classes:  \begin{enumerate}
\item [$\M_{9a}$]  of size $2^2 \cdot 3 \cdot 5 \cdot 13$, consisting of a pair of hyperplanes of $PG(3,3)$.

\item [$\M_{9b}$] of size $2 \cdot 3^4 \cdot 5 \cdot 13$, consisting of  hyperbolic quadrics of $PG(3,3)$.

\item [$\M_{9c}$] of size $  2^3 \cdot 3^4  \cdot 13$, consisting of elliptic quadrics of $PG(3,3)$.
\end{enumerate}
\end{lemma}

\subsection{$B_{j,9}$ for $j>9$}
\subsubsection*{}
Let $T$ be a $\M_{9a}$ configuration of $22$ points represented by a pair of hyperplanes $\pi_1, \pi_2$ intersecting in a line $L$. 
Since $j>9$, the possible cases for $S \in \fC_i(T)$, $i\leq 8$ are  listed as follows.
\begin{enumerate}
\item[(i)] $S=\M_{8a}$ consisting of a plane $\pi$ and a line $\ell$ not contained in $\pi$. In this case $|\M_{8a}(T)|=2\cdot 12=24$, as there are $2$ choices for $\pi$, and $12$ choices for $\ell$. Therefore, \[|\M_{8a}(T)|\cdot B_{j,9}(\M_{8a}(T))=24 \cdot (\tbinom{16}{j}-3\cdot \tbinom{13}{j-1}-42\cdot \tbinom{10}{j}-\tbinom{13}{j}).\]

\item[(ii)]  $S=\M_{8b}$. For a $\M_{8b}(T)$ configuration of $4$ concurrent lines $L_1, \dots, L_4$ through a point $Q$, no three of which are coplanar, the point $Q$  has to be on the line $L$, and $\pi_1$, $\pi_2$ each  contain two of the lines $\{L_1, \dots, L_4\}$. 
The point $Q$ on $L$ can be chosen in $4$ ways, and the lines $L_1, \dots, L_4$ through $Q$ can be chosen in $9$ ways. Thus, 
 $|\M_{8b}(T)|=4\cdot 9 =36$. Therefore, 
 \[|\M_{8b}(T)|\cdot B_{j,9}(\M_{8b}(T))=36\cdot (\tbinom{13}{j}-4\cdot \tbinom{10}{j}).\]
  
\item[(iii)]  $S=\M_{8c}$. The $\M_{8c}(T)$ configurations can be chosen in the following way: pick  a pair of points $\{P, P'\}$ on the line $L$. A pair of lines $L_1, L_2$ through $P$, and another pair of lines $L_1', L_2'$ through $P'$ such that $L_1, L_2, L_1', L_2'$ are distinct from $L$, and $L_i, L_i'$ are in the plane $\pi_i$ for $i \in \{1, 2\}$.  There are $\tbinom{4}{2}$ ways to pick $\{P, P'\}$ on $L$, and there are  $3^4$ ways to choose the four lines $\{L_1, L_1', L_2, L_2'\}$. Thus, $|\M_{8c}(T)|=486$. Therefore, \[|\M_{8c}(T)|\cdot B_{j,9}(\M_{8c}(T))=486\cdot (\tbinom{12}{j}-4\cdot \tbinom{10}{j}).\]

\item[(iv)] $S=\M_{8d}$ consisting of 
    $(\{L_1,L_2\},\{P_1,P_2,P_3\})$ as described in Lemma \ref{max8}. Such an $S$  can be chosen as follows:
Pick a pair of lines $L_1,L_2$ in one of the planes $\pi_1, \pi_2$ intersecting at a point $Q$ of $L$. There are $2 \cdot 4 \cdot 3 = 24$ ways to pick these. In the other plane, there are $3$ lines $L_1, L_2, L_3$ through $Q$ apart from $L$. 
There are $18$ ways to pick $P_i \in L_i\setminus\{Q\}$ for $i \in \{1, 2, 3\}$  such that $P_1, P_2, P_3$ are  non-collinear.  Thus, $|\M_{8c}(T)|=24 \cdot 18=432$. Therefore, \[|\M_{8d}(T)|\cdot B_{j,9}(\M_{8d}(T))=432 \cdot \tbinom{10}{j}.\]

\item[(v)] $S=\M_{7a}$ consisting of a plane $\pi$ and a point $P$ not on $\pi$. There are $2$ choices for $\pi$ namely $\pi_1, \pi_2$. Any of the $9$ points of the other plane not lying on $L$ can be chosen as $P$. Thus,    $|\M_{7a}(T)|=2\cdot 9=18$.Therefore, \[|\M_{7a}(T)|\cdot B_{j,9}(\M_{7a}(T))=18\cdot \tbinom{13}{j-1}.\]

\item[(vi)]  $S=\M_{7b}$. Let us consider a $\M_{7b}(T)$ configuration represented by $3$ concurrent non-coplanar lines $L_1,L_2,L_3$ through a point $Q$. There are two cases depending on whether $L \in \{L_1,L_2,L_3\}$ or not. 
\begin{enumerate}
    \item [(a)]
 In the former case, we may assume  $L=L_1$. The number of way to choose $Q$ on $L$ is $4$. Then the two other lines $L_2, L_3$ through $Q$ lie in  different planes, and can be chosen in $\tbinom{3}{1}^2 = 9$ ways. This case contributes $36$ to $|\M_{7b}(T)|$.\\

    \item [(b)] If $L \notin \{L_1,L_2,L_3\}$, then two of the lines $\{L_1,L_2,L_3\}$ lie in the same plane. The number of ways of  choosing $Q$ on $L$ is $4$, the number of ways choosing two lines through $Q$ in one of the plane $\pi_1, \pi_2$ is  $2 \cdot \tbinom{3}{2}=6$.  The remaining line through $Q$ in the other plane can be chosen in $\tbinom{3}{1}$ ways. Thus,  this case contributes  $4\cdot 2\cdot 3\cdot 3=72$ to $|\M_{7b}(T)|$.\\
\end{enumerate}

Thus, $|\M_{7b}(T)|=36+72=108$. Therefore, \[|\M_{7b}(T)|\cdot B_{j,9}(\M_{7b}(T))=108\cdot \tbinom{10}{j}.\]

\item[(vii)] $S=\M_{7c}$ represented by three lines $L_1,L_2,L_3$, where $L_2$ and $L_3$ are skew to each other and the line $L_1$ meets both the lines.  There are two cases depending on whether $L \in \{L_1,L_2,L_3\}$ or not. 

\begin{enumerate}
    \item [(a)] If  $L \in \{L_1,L_2,L_3\}$, say $L=L_3$. let $P_i=L\cap L_i$, $i=1,2$. The number of way to choose $P_1$ and $P_2$ is $4\cdot 3$. Then the two lines $L_1$ and $L_2$ can be chosen in $\tbinom{3}{1}^2$. Thus the number of ways of choosing $L_1,L_2,L_3$ is $4\cdot 3 \cdot \tbinom{3}{1}^2=108$.  Thus, this case contributes  $108$  to $|\M_{7c}(T)|$.\\
\item [(b)]  If $ L \notin \{L_1,L_2,L_3\}$,  then exactly one of the two pairs of lines $\{L_1,L_2\}$ and  $\{L_2,L_3\}$ are coplanar, say  $\{L_1,L_2\}$. In this case, the pair  $\{L_2,L_3\}$ can be chosen $4 \cdot 3 \cdot 3$ ways, and $L_1$ can be chosen in $9$ ways. Thus, this case contributes  $2 \cdot 4 \cdot 3 \cdot 3 \cdot 9=648$ to $|\M_{7c}(T)|$.\\
    \end{enumerate}

 Therefore, $|\M_{7c}(T)|=108+648=756$ and 
\[|\M_{7c}(T)|\cdot B_{j,9}(\M_{7c}(T))=756 \cdot \tbinom{10}{j}.\]

\item[(viii)]  $S=\M_{6a}$. In this case $|\M_{6a}(T)|=2$. Therefore, \[|\M_{6a}(T)|\cdot B_{j,9}(\M_{6a}(T))=2\cdot \tbinom{13}{j}.\]

\end{enumerate}

By \eqref{eq:B_j_r_alpha}, for $j>9$, we have,
\begin{align*}
\begin{split}
    B_{j,9}(\M_{9a})&=780\cdot(\tbinom{22}{j}-24 \cdot (\tbinom{16}{j}-3\cdot \tbinom{13}{j-1}-42\cdot \tbinom{10}{j}-\tbinom{13}{j})-36\cdot (\tbinom{13}{j}-4\cdot \tbinom{10}{j})\\
    &-486\cdot (\tbinom{12}{j}-4\cdot \tbinom{10}{j})-432 \cdot \tbinom{10}{j}-18\cdot \tbinom{13}{j-1}-108\cdot \tbinom{10}{j}-756 \cdot \tbinom{10}{j}-2\cdot \tbinom{13}{j}),
    \end{split}
\end{align*}
i.e. 
\begin{align*}
B_{j,9}(\M_{9a})=780\cdot(\tbinom{22}{j}-24 \cdot \tbinom{16}{j}+1800\cdot \tbinom{10}{j}-486\cdot \tbinom{12}{j}+54\cdot \tbinom{13}{j-1}-62\cdot \tbinom{13}{j}).
\end{align*}

\subsubsection*{}
Let $T$ be a $\M_{9b}$ configuration represented by a hyperbolic quadric $\mH$.  A hyperbolic quadric in $PG(3,3)$ can be written as a disjoint union of $4$ skew lines $H_1, \dots, H_4$ (known as a regulus), and also as a disjoint union of a dual set of $4$ skew lines $V_1, \dots, V_4$ (known as the opposite  regulus). The  set $T$ consists of the $16$ points $H_i \cap V_j$ for $1 \leq i, j \leq 4$.
Since $j>9$, the possible cases for $S \in \fC_i(T)$, $i\leq 8$ are  as follows.
\begin{enumerate}
    \item[(i)] $S=\M_{8c}$ consisting of $4$ lines $L_1, \dots, L_4$ with $L_i, L_j$ being skew if and only if $i \equiv j \mod 2$. Here,  $|\M_{8c}(T)|=\tbinom{4}{2}^2=36$, as $\{L_1, \dots, L_4\}$ consists of $2$ lines from $\{H_1, \dots, H_4\}$ and $2$ lines from $\{V_1, \dots, V_4\}$.Therefore, 
    \[|\M_{8c}(T)|\cdot B_{j,9}(\M_{8c}(T))=36 \cdot (\tbinom{12}{j}-4\cdot \tbinom{10}{j}).\]

    \item[(ii)] $S=\M_{8d}$ consisting of 
    $(\{L_1,L_2\},\{P_1,P_2,P_3\})$ as described in Lemma \ref{max8}. We must have 
    $|\{L_1, L_2\} \cap \{H_1, \dots, H_4\}| = |\{L_1, L_2\} \cap \{V_1, \dots, V_4\}| = 1$, and hence there are $16$ ways to pick $\{L_1, L_2\}$. If $L_1=H_4$ and $L_2=V_4$, the points $\{P_1, P_2, P_3\}$ must be $\{H_1 \cap V_{\sigma(1)}, H_2 \cap V_{\sigma(2)}, H_3 \cap V_{\sigma(3)}\}$ for some permutation $\sigma$ of $\{1, 2, 3\}$. i.e., there are $6$ ways to pick $\{P_1, P_2, P_3\}$. Thus,  $|\M_{8d}(T)|=16\cdot 6=96$. Therefore, \[|\M_{8d}(T)|\cdot B_{j,9}(\M_{8d}(T)) = 96\cdot \tbinom{10}{j}.\]

    \item[(iii)]  $S=\M_{7c}$ represented by three lines $L_1,L_2,L_3$, where $L_2$ and $L_3$ are skew to each other and the line $L_1$ meets both the lines. We must have $L_1 \in \{H_1, \dots, H_4\}$ and $L_2, L_3 \in \{V_1, \dots, V_4\}$ or  $L_1 \in \{V_1, \dots, V_4\}$ and $L_2, L_3 \in \{H_1, \dots, H_4\}$. Thus,
$|\M_{7c}(T)|=2 \cdot \tbinom{4}{2}\cdot \tbinom{4}{1}=48$. Therefore, \[|\M_{7c}(T)|\cdot B_{j,9}(\M_{7c}(T)) = 48\cdot \tbinom{10}{j}.\]
\end{enumerate}

By \eqref{eq:B_j_r_alpha}, for $j>9$, we have, \[B_{j,9}(\M_{9b})=\tbinom{16}{j}-36 \cdot (\tbinom{12}{j}-4\cdot \tbinom{10}{j})-96\cdot \tbinom{10}{j}-48\cdot \tbinom{10}{j}=\tbinom{16}{j}-36 \cdot \tbinom{12}{j}.\]

\subsubsection*{}
Let $T$ be a $\M_{9c}$ configuration represented by an elliptic quadric $\mQ_3 = V(\varphi_3)$ consisting of  $10$ points. The only  size $j$ subset of $T$ for $j >9$, is $T$ itself.
Therefore,  
\[B_{j,9}(\M_{9c})= \tbinom{10}{j}.\]
Finally,  for $j>9$,  
we have by \eqref{eq:B_j_r}:
\begin{equation*}
    \begin{split}
B_{j,9}&=  780\cdot(\tbinom{22}{j}-24 \cdot \tbinom{16}{j}+1800\cdot \tbinom{10}{j}-486\cdot \tbinom{12}{j}+54\cdot \tbinom{13}{j-1}-62\cdot \tbinom{13}{j})\\
&+ 10530\cdot (\tbinom{16}{j}-36\cdot \tbinom{12}{j})+8424\cdot \tbinom{10}{j},
\end{split}
\end{equation*}
i.e.,
\begin{equation}\label{BJ9}
    \begin{split}
B_{j,9}=  780\cdot\tbinom{22}{j}-8190\cdot \tbinom{16}{j}-48360\cdot \tbinom{13}{j}-14025960 \cdot \tbinom{12}{j}+1412424 \cdot \tbinom{10}{j}+42120\cdot \tbinom{13}{j-1}
\end{split}
\end{equation}

\section{Proof of Theorem \ref{main_theorem}}  \label{BTAT}
We recall from \S  \ref{Approach}, that we have defined $B_0(T)=(T^K-1)$. Using \eqref{eq:Bji},\eqref{eq:Bjj},\eqref{eq:BjK},\eqref{BJ3}-\eqref{BJ9} we calculate the values of $B_j(T)$:
\begin{align*}
 B_1(T)=& 40(T^9-1),\\
B_2(T)=& 780(T^8-1),\\
B_3(T)=&9880(T^7-1),\\
B_4(T)=&130(T^7+702T^6-703),\\ 
B_5(T)=& 936(5 T^6 +698 T^5 -703), \\
B_6(T)=& 780(109 T^5 + 4812 T^4 - 4921), \\
B_7(T)=& 1560 (2 T^5 + 663 T^4 + 11286 T^3 - 11951), \\
B_8(T)=&
585(T - 1)(241 T^3 + 16333 T^2 + 131461 T + 131461),\\
B_9(T) =&  520(T - 1)(55 T^3 + 5500 T^2 + 138016 T + 525844),\\
B_{10}(T) =& 104(T-1)(110T^3 + 8435T^2 + 399125T + 4410326), \\
 B_{11}(T)=&3120(T - 1)(T^3 + 100T^2 + 6211T + 171775), \\
B_{12}(T)=&260(T - 1)(2T^3 + 326T^2 + 29837T + 1921619), \\
B_{13}(T)=&40(T - 1)(T^3 + 352T^2 + 61426T + 9659872), \\
B_{14}(T)=&1080(T - 1)(T^2 + 508T + 230582), \\
B_{15}(T)=&18720(T-1)(4T + 7103),\\
B_{16}(T)=&1170(4T + 49739)(T - 1),\\
B_{17}(T)=&20540520(T-1),\\
B_{18}(T)=&5705700(T-1),\\
B_{19}(T)=&1201200(T-1),\\
B_{20}(T)=&180180(T-1),\\
B_{21}(T)=&17160(T-1),\\
B_{22}(T)=&780(T-1).
\end{align*}

Having calculated the polynomials $B_j(T)$, we use them in \eqref{eq:AB1}, to determine the polynomials $A_j(T)$:
\begin{theorem}\label{AJT}
The extended weight enumerator of $C_3$ is given by \[W(X,Y,T)=X^{40}+(T-1)\sum_{i=1}^{40} a_i(T)X^{40-i}Y^i,\] where the polynomials $a_i(T)$ are as follows:
\begin{align*}
a_{18}(T)=& 780,\\
a_{24}(T)= &1170(4 T - 3),\\
a_{26}(T)=&1080(T - 3)(T - 9),\\
a_{27}(T)=&40(T^3 - 26 T^2 + 442 T - 884),\\
a_{28}(T)=&189540(T - 3),\\
a_{30}(T)= &936(100 T^2 - 600 T + 909,)\\
a_{31}(T)=&505440(T - 3)(T - 6),\\
a_{32}(T)= &5265(T - 3)( 17 T^2 - 160 T + 513),\\
a_{33}(T)= &3120(T - 3)(T^3 + 84 T^2 - 972 T + 3000),\\
a_{34}(T)= &63180(T - 3)( T^3 - 5 T^2 - 45 T + 249),\\
a_{35}(T)=& 936(T - 3)( 5 T^4 + 243 T^3 - 6408 T^2 + 51654 T - 142020),\\
a_{36}(T)= &130(T^6 + 523 T^5 - 15635 T^4 + 199500 T^3 - 1349544 T^2 + 4639614 T - 6020826),\\
a_{37}(T)= &9360(T - 3)(T^5 - 30 T^4 + 405 T^3 - 3132 T^2 + 13932 T - 27864),\\
a_{38}(T)= &780(T - 3)(T^6 - 33 T^5 + 507 T^4 - 4698 T^3 + 28242 T^2 - 105678 T + 189054),\\
a_{39}(T)=& 40(T - 3)(T^7 - 35 T^6 + 585 T^5 - 6096 T^4 + 43320 T^3 - 214344 T^2 + 695448 T - 1128816),\\
\begin{split} 
a_{40}(T) =& (T - 3) \left( T^8 - 36 T^7 + 633 T^6 - 7110 T^5 + 56241 T^4 - 325134 T^3 \right. \\ 
 & \qquad \qquad \left. + 1371006 T^2 - 3936114 T + 5856786 \right). \end{split}
\end{align*}
\end{theorem}

Having obtained the extended weight enumerator $W(Z;T)$ of $C_3$ over $\bF_3$, we are in a position to compute the quantities $W(Z;q^j)$ for $q=3$.
We calculate the $r$-GWE of $C_3$ listed in Theorem \ref{main_theorem}, by  using the quantities $W(Z;q^j)$ in the result of Theorem \ref{CO}.

\begin{thebibliography}{1}

\bibitem{GGG}
J.~W.~P. Hirschfeld and J.~A. Thas, \emph{General {G}alois geometries}, Springer Monographs in Mathematics, Springer, London, 2016. \MR{3445888}

\bibitem{JV2}
Trygve Johnsen and Hugues Verdure, \emph{Higher weight spectra of {V}eronese codes}, IEEE Trans. Inform. Theory \textbf{66} (2020), no.~6, 3538--3546. \MR{4115115}

\bibitem{JV3}
\bysame, \emph{Higher weight spectra of codes from {V}eronese threefolds}, J. Pure Appl. Algebra \textbf{225} (2021), no.~7, Paper No. 106609, 11. \MR{4217930}

\bibitem{JP}
Relinde Jurrius and Ruud Pellikaan, \emph{Codes, arrangements and matroids}, Algebraic geometry modeling in information theory, Ser. Coding Theory Cryptol., vol.~8, World Sci. Publ., Hackensack, NJ, 2013, pp.~219--325. \MR{3288286}

\bibitem{LG}
Gilles Lachaud, \emph{Projective {R}eed-{M}uller codes}, Coding theory and applications ({C}achan, 1986), Lecture Notes in Comput. Sci., vol. 311, Springer, Berlin, 1988, pp.~125--129. \MR{960714}

\end{thebibliography}
\bibliographystyle{amsplain}

\end{document}